\documentclass[10pt,oneside]{article}
\usepackage{mathrsfs}
\usepackage{amsmath,amsfonts,amssymb,amsthm}
\usepackage{caption}
\usepackage{subfigure}
\usepackage{cite}
\usepackage{color}
\usepackage{graphicx}
\usepackage{subfigure}
\usepackage{float}
\usepackage{paralist}
\usepackage[colorlinks,linkcolor=red,citecolor=blue]{hyperref}
\usepackage{authblk}
\usepackage[T1]{fontenc}
\newtheorem{thm}{Theorem}[section]
\newtheorem{lem}[thm]{Lemma}
\newtheorem{prop}[thm]{Proposition}
\newtheorem{rem}{Remark}[section]

\def\QEDopen{{\setlength{\fboxsep}{0pt}\setlength{\fboxrule}{0.2pt}\fbox{\rule[0pt]{0pt}{1.3ex}\rule[0pt]{1.3ex}{0pt}}}} %
\def\QED{\QEDopen} %
\def\endproof{\hspace*{\fill}~\QED\par\endtrivlist\unskip}%
\def\bma#1\ema{{\allowdisplaybreaks\begin{split}#1\end{split}}}

\numberwithin{equation}{section}

\begin{document}	
	
	\newpage
	\title{Global well-posedness and optimal time decay rates of solutions to the pressureless Euler-Navier-Stokes system }
	\author[a,b]{ Feimin Huang   \thanks{E-mail: fhuang@amt.ac.cn(F.-M. Huang)}}
	\author[a]{ Houzhi Tang   \thanks{E-mail: houzhitang@amss.ac.cn(H.-Z. Tang)}}
	\author[c]{Weiyuan Zou \thanks{E-mail: zwy@amss.ac.cn(W.-Y. Zou)}}
	\affil[a]{Academy of Mathematics and Systems Science, Chinese Academy of Sciences, Beijing 100190, P. R. China}
	\affil[b]{School  of  Mathematical  Sciences,  University  of  Chinese  Academy  of  Sciences,  Beijing 100049, P. R. China }
	\affil[c]{College of Mathematics and Physics,
		Beijing University of Chemical Technology, Beijing 100029, P. R. China} 
	
	\date{}
	\renewcommand*{\Affilfont}{\small\it}	
	\maketitle
	\begin{abstract}
		In this paper, we present a new framework for the global well-posedness and large-time behavior of a two-phase flow system,  which consists of the pressureless Euler equations and incompressible Navier-Stokes equations coupled through the drag force. To overcome the difficulties arising from the absence of the pressure term in the Euler equations, we establish the time decay estimates of the high-order derivative of the velocity to obtain uniform estimates of the fluid density. The upper bound decay rates are obtained by designing a new functional and the lower bound decay rates are achieved by selecting specific initial data. Moreover, the upper bound decay rates are the same order as the lower one. Therefore, the time decay rates are optimal. When the fluid density in the pressureless Euler flow vanishes, the system is reduced into an  incompressible Navier-Stokes flow. In this case, our works coincide with the classical results by  Schonbek \cite{M.S3} [JAMS,1991], which can be regarded as a generalization from a single fluid model to the two-phase fluid one. 
	\end{abstract}
	\noindent{\textbf{Key words:}\
		Pressureless Euler-Navier-Stokes system, large-time behavior, optimal decay rates.\\
		\textbf{2010 MR Subject Classification:}\  35B40, 35B65, 76N10\\
	\section{Introduction}
	This paper is concerned with the global well-posedness and large-time behavior of a coupled hydrodynamic system consisting of the pressureless Euler equations and incompressible Navier-Stokes equations coupled through the drag force,  which reads as 
	\begin{equation}\label{Main1}
		\left\{
		\begin{aligned}
			&\rho_t+\text{div}(\rho u)=0,\\
			&(\rho u)_t+\text{div}(\rho u\otimes u)=-\rho(u-v),\\
			&v_t+v\cdot\nabla v+\nabla P=\Delta v+\rho(u-v),\\
			&\text{div}v=0,
		\end{aligned}
		\right.
	\end{equation}
	with the initial data 
	\begin{equation}\label{ID}
		(\rho,u,v)|_{t=0}=(\rho_0,u_0,v_0).
	\end{equation}
	The unknown functions $\rho=\rho(t,x)$ and $u=u(t,x)$ denote the density and  velocity of the pressureless Euler fluid flow; $v=v(t,x)$ and $P=P(t,x)$ represent the velocity and   pressure of the incompressible Navier-Stokes fluid flow.
	
	This coupled pressureless Euler-Navier-Stokes (E-NS) system can be formally derived from the Vlasov-Navier-Stokes system, which describes the behavior of a large cloud of particles interacting with the incompressible fluid, in the case of mono-kinetic particle distributions. The details of the derivation can be referred to \cite{ChoiJung2021JMFM}. The main system \eqref{Main1} is closely related to the kinetic-fluid models, which have received increasing attention due to its wide range of applications, for instance, including medicine, biotechnology, combustion in diesel engines, and atmospheric pollution \cite{BJM2005,BDR2003CMS,O1981,IMF2006,William1958}.
	
	There has been important progress made recently on the well-posedness and dynamic behaviors of the solutions to the Euler-Navier-Stokes system and related models, refer to \cite{ls2003,GWZ2023,yz2019,MV2008,CYZ2016M3}  and the references therein. Concerning the Euler system coupled with compressible Navier-Stokes equations, Choi \cite{ChoiSIAM2016} proved the global well-posedness and exponential time decay rates of the classical solution in the periodic domain. Later, Wu-Zhang-Zou \cite{Wu2020} obtained the large-time behavior in $\mathbb{R}^3$ and Wu-Zhou \cite{Wu2021} studied the pointwise space-time behavior of the Cauchy problem to this system.
	
	Different from the Euler-Navier-Stokes system  mentioned above, the Euler system of \eqref{Main1} is pressureless. It is well-known that the pressureless Euler equations develop a finite-time formation of singularities, for instance, $\delta$-shock \cite{HW2001,EPD1996,BG1998,ES1996}. Owing to the drag force $\big\{\rho(u-v)\big\}$ from the Navier-Stokes flow, the densities of the pressureless Euler part in \eqref{Main1} become regular. Indeed, it was shown in \cite{ChoiJung2021JMFM,YoungChoil2022, Ha2014M3} that the density is smooth. In fact, under the assumption that 
	\begin{equation}\label{YJMFM} \|\rho_0\|_{H^s(\mathbb{R}^n)}+\|u_0\|_{H^{s+2}(\mathbb{R}^n)}+\|v_0\|_{H^{s+1}(\mathbb{R}^n)}+\|v_0\|_{L^1(\mathbb{R}^n)} \ll1,\quad \rho_0(x)>0,
	\end{equation}
	Choi and Jung \cite{ChoiJung2021JMFM} proved the global well-posedness of system \eqref{Main1}-\eqref{ID} in the whole space and obtained the following time decay rate
	\begin{equation}\label{Y1}
		\|u(t)\|_{H^{s+1}(\mathbb{R}^n)}+\|v(t)\|_{H^s(\mathbb{R}^n)}\leq C(1+t)^{-\alpha},\quad \forall~ 0<\alpha<\frac{n}{4}.
	\end{equation} 
	In this paper, under the assumption weaker than \eqref{YJMFM} that
	\begin{equation} \|\rho_0\|_{H^s(\mathbb{R}^n)}+\|u_0\|_{H^{s+2}(\mathbb{R}^n)}+\|v_0\|_{H^{s+1}(\mathbb{R}^n)}\ll1,\quad \rho_0(x)>0,
	\end{equation}
	we prove that the time decay rate in \eqref{Y1} holds for $\alpha=n/4$. 
	The precise statement of 
	the first main theorem is given as
	\begin{thm}\label{Th1}
		Let the integers $n\geq 3$ and $s\geq 2[\frac{n}{2}]+1$. Assume the initial data satisfy $\rho_0(x)>0$, $\rho_0(x)\in H^s(\mathbb{R}^n)\cap L^1(\mathbb{R}^n)$, $u_0\in H^{s+2}(\mathbb{R}^n)$, $v_0\in H^{s+1}(\mathbb{R}^n)\cap L^1(\mathbb{R}^n)$ and ${\rm{div}} v_0=0$. There exists a sufficiently small constant $\delta_0>0$ such that if 
		\begin{equation}
			\|\rho_0\|_{H^s(\mathbb{R}^n)}+\|u_0\|_{H^{s+2}(\mathbb{R}^n)}+\|v_0\|_{H^{s+1}(\mathbb{R}^n)}\leq \delta_0,
		\end{equation}
		then the pressureless E-NS system \eqref{Main1}-\eqref{ID} admits a global and unique  classical solution $(\rho,u,v)$,
		$$ 
		(\rho, u,v)(t,x)\in \mathcal{C}([0,+\infty);H^s(\mathbb{R}^n))\times \mathcal{C}([0,+\infty);H^{s+2}(\mathbb{R}^n))\times \mathcal{C}([0,+\infty);H^{s+1}(\mathbb{R}^n))
		$$
		satisfying $\rho(t,x)>0$ for all $(t,x)\in [0,+\infty)\times \mathbb{R}^n$ and 
		\begin{equation}\label{T1}
			\sup_{t\geq 0}(\|\rho(t,x)\|_{H^s((\mathbb{R}^n))}+\|u(t,x)\|_{H^{s+2}(\mathbb{R}^n)}+\|v(t,x)\|_{H^{s+1}(\mathbb{R}^n)})\leq C\delta_0.
		\end{equation}
		Meanwhile, it holds for $t\geq 0$ that 
		\begin{equation}\label{OPT}
			\|u(t)\|_{H^{s+1}(\mathbb{R}^n)}\leq C\mathcal{I}_0(1+t)^{-\frac{n}{4}},\quad \|v(t)\|_{H^{s+1}(\mathbb{R}^n)}\leq C\mathcal{I}_0(1+t)^{-\frac{n}{4}},
		\end{equation}
		where the positive constant $C$ is independent of time and $\mathcal{I}_0>0$ is given by
		\begin{equation}\label{mathcalI0}
			\mathcal{I}_0=\delta_0+\|\rho_0\|_{L^1}+\|v_0\|_{L^1},
		\end{equation}
		which is not necessarily small.
	\end{thm}	
	Indeed, the lower bound of $\|v\|_{L^2}$ is achieved as the same order as the upper one given in  \eqref{OPT} by selecting special initial data. The precise statement of the work is given as 
	\begin{thm}\label{Th2}
		Assume the same conditions in Theorem \ref{Th1} hold. There exists a sufficiently small constant $\delta_0>0$ such that if 
		\begin{equation} \label{th2a}
			\|\rho_0\|_{H^s(\mathbb{R}^n)}+\|u_0\|_{H^{s+2}(\mathbb{R}^n)}+\|v_0\|_{H^{s+1}(\mathbb{R}^n)}+\|v_0\|_{L^1(\mathbb{R}^n)}\leq \delta_0,
		\end{equation}
		and the Fourier transform of the initial velocity $\hat{v}_0(\xi)$ satisfies
		\begin{equation}\label{A1}
			\inf_{|\xi|\leq 1}|\hat{v}_0(\xi)|\geq \delta_0^{\frac{3}{2}},
		\end{equation} 
		then it holds for large time that 
		\begin{equation}\label{optimal}
			\frac{1}{2}\delta_0^{\frac{3}{2}}(1+t)^{-\frac{n}{4}}\leq \|v(t)\|_{L^2(\mathbb{R}^n)} \leq C\mathcal{I}_0(1+t)^{-\frac{n}{4}},
		\end{equation}
		where the constant $C>0$ is independent of time.
	\end{thm}
	\begin{rem}
		In Theorem \ref{Th1},  the smallness condition  for $\|v_0\|_{L^1}$ in \eqref{YJMFM} is removed and the decay rate for $\|\nabla^{s+1}v\|_{L^2}$ is further derived in \eqref{OPT} while only $\|\nabla^sv\|_{L^2}$ is obtained in \eqref{Y1}.
	\end{rem}
	\begin{rem}
		From \eqref{optimal}, the decay rate of $\|v\|_{L^2}$ obtained in Theorem \ref{Th1} and \ref{Th2} is optimal.
	\end{rem}
	
	\begin{rem}
		In the special case of $\rho=0$, Theorem \ref{Th2} coincides with the resut \cite{M.S3} [JAMS,1991].
	\end{rem}
	
	\begin{rem}
		Due to the absence of the pressure term in the Euler equations on \eqref{Main1}, it is not expected to obtain  the time decay rate of the fluid density $\rho(t,x)$.
	\end{rem}
	
	Now we sketch the main ideas of Theorem \ref{Th1}. Let us first recall \cite{ChoiJung2021JMFM} in which the authors derived the basic energy inequality 
	\begin{equation}\label{32703}
		\frac{d}{dt} E(t)+D(t)\leq 0, \quad E(t)= \int_{\mathbb{R}^n} \rho|u|^2dx+\|v\|_{L^2}^2,~D(t)=\int_{\mathbb{R}^n}\rho|u-v|^2dx+\|\nabla v\|_{L^2}^2.
	\end{equation}
	Then by very careful estimates, they obtained that for any $\alpha \in (0,n/4)$, 
	\begin{equation}\label{choi3.1}
		E(t)(1+t)^{2\alpha}+\int_0^t(1+\tau)^{2\alpha} D(\tau)d\tau\leq C(E(0)+\|v_0\|_{L^1}^2),
	\end{equation}
	which gives the decay rates in \eqref{Y1} provided that $E(0)+\|v_0\|_{L^1}^2$ is sufficiently small. 
	
	In order to show that the decay rate in \eqref{Y1} still holds for $\alpha=n/4$, we observe from the basic energy inequality \eqref{32703} that 
	
	\begin{equation}\label{32201}
		\int_0^t\|\rho(u-v)(\tau)\|_{L^1}d\tau\leq 2\|\rho_0\|_{L^1}^{\frac{1}{2}}\big(10\varepsilon_0 M(t)+2\varepsilon_1E(0)\big)^{\frac{1}{2}},
	\end{equation}
	holds for some constants $\varepsilon_0$ and $\varepsilon_1$ through an elaborate weighted estimate, where  $M(t)$ represents the energy functional given in \eqref{MTT}.  Noting that $\eqref{Main1}_{3}$ and  $\eqref{Main1}_{4}$ are similar to the incompressible Navier-Stokes equations with source term, we adopt Fourier splitting method as in \cite{M.S1, zgh2011, M.S3} to obtain
	\begin{equation}\label{32102}
		\|v(t)\|_{L^2}\leq C\mathcal{I}_0(1+t)^{-\frac{n}{4}},\quad 
		\mathcal{I}_0=\delta_0+\|\rho_0\|_{L^1}+\|v_0\|_{L^1}.
	\end{equation}
	
	Noting also that $\eqref{Main1}_{2}$ is equivalent to  
	\begin{equation}\label{newu11}
		u_t+u\cdot\nabla u+u-v=0,	
	\end{equation}
	if the density $\rho(t,x)$ is away from vacuum. Then, it yields the time decay rate of $\|u\|_{L^2}$ by the damping structure as well as \eqref{32102}. With the decay rates of $\|(u,v)\|_{L^2}$ in hand, the decay estimates of $\|\nabla(u,v)\|_{H^{s}}$ can be derived in a simialr process. Hence we obtain  
	\begin{equation}\label{optimalD}
		\|(u,v)(t)\|_{H^{s+1}(\mathbb{R}^n)}\leq C\mathcal{I}_0(1+t)^{-\frac{n}{4}}.
	\end{equation}
	It should be remarked that $\mathcal{I}_0$ is not small, thus we can remove the smallness condition for $\|v_0\|_{L^1}$. Thanks to \eqref{optimalD}, we apply $L^2$ energy method to obtain that
				$$
				\|\rho(t)\|_{H^s}^2+	\|u(t)\|_{H^{s+2}}^2+\|v(t)\|_{H^{s+1}}^2+\int_0^t(\| \nabla u(\tau)\|_{H^{s+1}}^2+\|\nabla v(\tau)\|_{H^{s+1}}^2)d\tau
				\leq C\delta_0^2,
				$$
				which closes the a priori assumption \eqref{a-priori est}.
				
				The key point to prove  Theorem \ref{Th2} is the lower bound decay estimate of $\|v\|_{L^2}$. To achieve the goal, we first consider the heat equation $w_t-\Delta w=0$ with the initial data $w_0=v_0$. It is easy to verify from \eqref{A1} that $\delta_0^{\frac{3}{2}}(1+t)^{-\frac{n}{4}}$ is the lower bound of $\|w\|_{L^2}$. Then, under the assumptions in \eqref{th2a}, we can apply \eqref{choi3.1} to show
				\begin{equation}\label{041203}
					\int_0^t\|\rho(u-v)(\tau)\|_{L^1}d\tau\leq C\delta_0^2.
				\end{equation}
				Note that the main part of the system $\eqref{Main1}_3$-$\eqref{Main1}_4$ for $v$ is the incompressible Navier-Stokes equations whose large time behavior is the same as the one of the heat equation, we can further prove that 
				$$\|(v-w)(t)\|_{L^2}\leq C\delta_0^2(1+t)^{-\frac{n}{4}}. 
				$$
				Then by the triangle inequality and the smallness of $\delta_0$ imply the lower bound of $\|v\|_{L^2}$  is the same order as $\|w\|_{L^2}$, i.e. \eqref{optimal}.

				The rest of the paper is organized as follows. In Section \ref{Sec2}, we provide some notations and auxiliary lemmas. In Section \ref{S3}, we study the global well-posedness and upper bound of the decay rates to the pressureless E-NS system. In Section \ref{S4}, we establish the optimal decay rate of the pressureless E-NS system.

				\section{Preliminaries}\label{Sec2}
				\subsection{Notations}
				\hspace{2em}In this section, we first introduce the notation and conventions used throughout the paper.
				$L^p(\mathbb{R}^n)$ and $W^{k,p}(\mathbb{R}^n)$ denote the usual Lebesgue and Sobolev space on $\mathbb{R}^n$, with norms $\|\cdot\|_{L^p}$ and $\|\cdot\|_{W^{k,p}}$, respectively.  When $p=2$, we denote $W^{k,p}(\mathbb{R}^n)$ by $H^k(\mathbb{R}^n)$ with the norm $\|\cdot\|_{H^k}$, and set
				$$
				\|u\|_{H^k(\mathbb{R}^n)}=\|u\|_{H^k},\quad \|u\|_{L^p(\mathbb{R}^n)}=\|u\|_{L^p}.
				$$
				The symbol $[ \cdot ]$ represents the floor function, i.e., $[n/2]$ is the greatest integer less than or equal to $n/2$. We denote by $C$ a generic positive constant which may vary in different estimates. For an integer $k$,  the symbol $\nabla^k$ denotes the summation of all terms $D^\ell=\partial_1^{\ell_1}\partial_2^{\ell_2}...\partial_n^{\ell_n}$ with the multi-index $\alpha$ satisfying $|\ell|=\ell_1+\ell_2+...+\ell_n=k$. For a function $f$, $\|f\|_{X}$ represents the norm of $f$ on $X$, $\|(f,g)\|_{X}$ denotes $\|f\|_{X}+\|g\|_{X}$. The Fourier transform of $f$ is denoted by $\hat{f}$ or $\mathscr{F}[f]$ satisfying
				\begin{equation}\hat{f}(\xi)
					=\mathscr{F}[f](\xi)
					=(2\pi)^{-\frac{n}{2}}\int_{\mathbb{R}^n} f(x)e^{-ix\cdot\xi}dx,\quad \xi\in\mathbb{R}^n.
				\end{equation}
				The inverse Fourier transform of $f$ is denoted by $\check{f}$ or $\mathscr{F}^{- 1}[f]$ such that 
				\begin{equation}
					\check{f}(x)=\mathscr{F}^{- 1}[f](x)=(2\pi)^{-\frac{n}{2}}\int_{\mathbb{R}^n}  f(\xi)e^{i\xi\cdot x}d\xi,\quad x\in\mathbb{R}^n.
				\end{equation}
				\subsection{Auxiliary lemmas}
				In the next part, we review some elementary inequalities and important lemmas that are used extensively in this paper. We first recall the following commutator estimates and Sobolev inequality, see \cite[Lemma 3.4]{mb2002}.
				\begin{lem}\label{lem1}
					(1) For any integer $k\geq1$ and the pair of functions $f,g\in H^k(\mathbb{R}^n)\cap L^{\infty}(\mathbb{R}^n)$, we have
					\begin{equation}
						\|\nabla^k(fg)\|_{L^2}\leq C \|f\|_{L^{\infty}}\|\nabla^kg\|_{L^2}+C\|\nabla^kf\|_{L^2}\|g\|_{L^{\infty}}.
					\end{equation}
					Moreover if $\nabla f\in L^{\infty}(\mathbb{R}^n)$, we obtain
					\begin{equation}
						\big\|\nabla^{k}(fg)-f\nabla^{k}g\big\|_{L^2}\leq C \|\nabla f\|_{L^{\infty}}\|\nabla^{k-1}g\|_{L^2}+C\|\nabla^kf\|_{L^2}\|g\|_{L^{\infty}}.
					\end{equation}
					(2) Let integers $n\geq 3$ and $s\geq 2\left[\frac{n}{2}\right]+1.$ If $f\in H^{s}(\mathbb{R}^n)$, then it holds the inequality 
					\begin{eqnarray}\label{L1}
						\|f\|_{L^{\infty}}\leq C\|\nabla f\|_{H^{[\frac{n}{2}]}}\leq C\|\nabla f\|_{H^{s-2}}.
					\end{eqnarray}
				\end{lem}
				
				Let $w(t,x)$ be the solution of the heat equation with initial data $w(0,x)=v_0$ in $(t,x)\in \mathbb{R}_+\times\mathbb{R}^n$, which reads
				\begin{equation}\label{we}
					\left\{
					\begin{aligned}
						&	w_t-\Delta w=0,\\	
						&w(0,x)=v_0(x).
					\end{aligned}
					\right.
				\end{equation}
				\begin{lem}\label{lem2}
					Let integers $n\geq 3$ and $s\geq 2[\frac{n}{2}]+1$. Assume $w(t,x)$ is the solution of \eqref{we},  then it holds for $t\geq 0$ that
					$$
					\|w(t,x)\|_{L^\infty}\leq C(1+t)^{-\frac{n}{2}}(\|v_0\|_{L^1}+\|\nabla v_0\|_{H^{s-2}}).
					$$
				\end{lem}
				\begin{proof}
					After a simple calculation, we obtain the explicit formula of $w(t,x)$ satisfying
					$$
					w(t,x)=(4\pi t)^{-\frac{n}{2}}\int_{\mathbb{R}^n} e^{-\frac{|x-y|^2}{4t}}v_0(y)dy.
					$$
					Consequently, we have the following estimate for large-time $t>1$ that 
					\begin{equation}\label{G1}
						\begin{aligned}
							\|w(t,x)\|_{L^\infty}\leq 	(4\pi t)^{-\frac{n}{2}} \int_{\mathbb{R}^n}|v_0(y)|dy\leq 	(4\pi t)^{-\frac{n}{2}}\|v_0\|_{L^1}\leq C(1+t)^{-\frac{n}{2}}\|v_0\|_{L^1}.
						\end{aligned}
					\end{equation}
					When $0\leq t\leq 1$, by energy method, we have for $s\geq [\frac{n}{2}]+2$ that 
					\begin{equation}\label{G2}
						\|w(t,x)\|_{L^\infty}\leq C\|\nabla w\|_{H^{s-2}}\leq C\|\nabla v_0\|_{H^{s-2}}\leq C(1+t)^{-\frac{n}{2}}\|\nabla v_0\|_{H^{s-2}}.
					\end{equation}
					The combination of \eqref{G1} and \eqref{G2} yields 
					$$
					\|w(t,x)\|_{L^\infty}\leq C(1+t)^{-\frac{n}{2}}(\|v_0\|_{L^1}+\|\nabla v_0\|_{H^{s-2}}).
					$$
					This completes the proof of this lemma.
				\end{proof}
				
				Now we are in a position to establish the lower bound of the time decay rate of $w(t,x)$ in the following lemma.
				\begin{lem}\label{lem3}
					Assume the Fourier transform of the initial velocity $\hat{v}_0(\xi)$ satisfies
					\begin{equation}\label{A2}
						\inf_{|\xi|\leq 1}|\hat{v}_0(\xi)|\geq \delta_0^{\frac{3}{2}}.	
					\end{equation} 
					Then, it holds 
					\begin{equation}
						\|w(t,x)\|_{L^2}\geq \delta_0^{\frac{3}{2}} (1+t)^{-\frac{n}{4}}.
					\end{equation}
				\end{lem}
				\begin{proof}
					In terms of Fourier transform on \eqref{we}, we have by \eqref{A2} that 
					$$
					\begin{aligned}
						\int_{\mathbb{R}^n}|w(t,x)|^2dx
						&=\int_{\mathbb{R}^n}|\hat{w}(t,\xi)|^2d\xi
						=\int_{\mathbb{R}^n}|\hat{v}_0(\xi)|^2e^{-|\xi|^2t}d\xi\\
						&\geq \int_{|\xi|\leq 1}|\hat{v}_0(\xi)|^2e^{-|\xi|^2t}d\xi\\
						&\geq\delta_0^3\int_{|\xi|\leq 1}e^{-|\xi|^2t}d\xi\\
						&\geq \delta_0^3\int_{|\xi|\leq 1}e^{-|\xi|^2(1+t)}d\xi.
					\end{aligned}
					$$
					Let $y_i=\sqrt{(1+t)}\xi_i ,~i=1,2,...,n$. We obtain 
					$$
					\begin{aligned}
						\int_{\mathbb{R}^n}|w(t,x)|^2dx	
						&\geq \delta_0^3(1+t)^{-\frac{n}{2}}\int_{|y|\leq \sqrt{1+t}}e^{-|y|^2}dy\\
						&\geq \delta_0^3(1+t)^{-\frac{n}{2}}\int_{|y|\leq 1}e^{-|y|^2}dy\\
						&\geq\delta_0^3(1+t)^{-\frac{n}{2}},
					\end{aligned}
					$$
					where we use the fact that 
					$$
					\int_{|y|\leq 1}e^{-|y|^2}dy\geq 4\pi \int_0^1 e^{-r^2}r^2dr \approx 2.38>1.
					$$
					It is shown that 
					$$
					\|w(t,x)\|_{L^2}\geq\delta_0^{\frac{3}{2}} (1+t)^{-\frac{n}{4}}.
					$$
					This completes the proof of this lemma.
				\end{proof}
				\section{Global well-posedness and upper bound of the decay rates}
				\label{S3}
				In this section, we will prove the global existence and uniqueness of the classical solution by using Fourier splitting method and energy method. As a by-product, we obtain the upper bound of the time decay rates. 
				
				Firstly, we present the local-time existence and uniqueness of classical solutions to our system \eqref{Main1}-\eqref{ID}.
				\begin{thm}\label{local}
					Let $n\geq 3$ and $s\geq 2[\frac{n}{2}]+1$. Assume the initial data satisfy $\rho_0(x)>0$, $\rho_0(x)\in H^s(\mathbb{R}^n)\cap L^1(\mathbb{R}^n)$, $u_0\in H^{s+2}(\mathbb{R}^n)$, $v_0\in H^{s+1}(\mathbb{R}^n)\cap L^1(\mathbb{R}^n)$ and ${\rm{div}}v_0=0$. There exists a small positive constant $\epsilon_0<\delta_0$ such that if 
					\begin{equation}
						\|\rho_0\|_{H^s}+\|u_0\|_{H^{s+2}}+\|v_0\|_{H^{s+1}}\leq 	\epsilon_0,
					\end{equation}
					then the pressureless E-NS system  \eqref{Main1}-\eqref{ID} has a unique solution $(\rho,u,v)$,
					$$
					(\rho, u,v)\in \mathcal{C}([0,T_0];H^s(\mathbb{R}^n))\times \mathcal{C}([0,T_0];H^{s+2}(\mathbb{R}^n))\times \mathcal{C}([0,T_0];H^{s+1}(\mathbb{R}^n))
					$$
					satisfying $\rho(t,x)>0$ for all $(t,x)\in [0,T_0]\times \mathbb{R}^n$ and 
					$$
					\sup_{0\leq t\leq T_0}(\|\rho(t,x)\|_{H^s}+\|u(t,x)\|_{H^{s+2}}+\|v(t,x)\|_{H^{s+1}})\leq \delta_0,
					$$
					where $T_0>0$ is a fixed constant.
				\end{thm}
				\begin{proof}
					The theory of local existence for each equation in \eqref{Main1} is well-developed in the $H^s$ Sobolev space. Thus we omit the details of the proof, seeing \cite{Choi2016JDE,Ha2014M3} for details.
				\end{proof}
				\vskip 4mm
				
				For notation simplicity, we denote the energy $\mathcal{Z}(t)$ and disspation $\mathcal{D}(t)$ below
				\begin{equation}\label{Dt}
					\begin{aligned}
						&\mathcal{Z}(t)\triangleq \Big(\|\rho(t)\|_{H^s}^2+\|u(t)\|_{H^{s+2}}^2+\|v(t)\|_{H^{s+1}}^2\Big)^{\frac{1}{2}},\\
						& \mathcal{D}(t)\triangleq \Big(\|\nabla u(t)\|_{H^{s+1}}^2+\|\nabla v(t)\|_{H^{s+1}}^2+\|(u-v)(t)\|_{L^2}^2)^{\frac{1}{2}}.
					\end{aligned}
				\end{equation} 
				
				Then, we extend the short-time classical solution to be a global one by establishing the uniform estimates. Hence, it is natural to provide the  a priori assumption for sufficiently small $\delta>0$
				\begin{equation}\label{a-priori est}
					\sup_{t\geq 0}\mathcal{Z}(t)\leq \delta,\quad \rho(t,x)>0.
				\end{equation}
				
				We prove the time decay rates of the solution $(u,v)$ to the system \eqref{Main1}-\eqref{ID} in the following proposition.
				\begin{prop}\label{p2}
					Assume the conditions in Theorem \ref{Th1} hold. Let $(\rho,u,v)(t,x)$ be the classical solution of the system \eqref{Main1}-\eqref{ID} satisfying the a priori assumption \eqref{a-priori est}. Then, it holds for $t\geq 0$ that 
					\begin{equation}\label{Udeacy}
						\|u(t)\|_{L^2}\leq C\mathcal{I}_0(1+t)^{-\frac{n}{4}},\quad \|v(t)\|_{L^2}\leq C\mathcal{I}_0(1+t)^{-\frac{n}{4}},
					\end{equation}
					where the positive constant $C$ is independent of time and $\mathcal{I}_0$ is given by \eqref{mathcalI0}.
				\end{prop}
				\begin{proof}
					Multiplying $\eqref{Main1}_2$ by $u$ and integrating the resulting equation in $\mathbb{R}^n$ yield
					\begin{equation}\label{T21}
						\frac{1}{2}\frac{d}{dt}\int_{\mathbb{R}^n}\rho|u|^2dx=-\int_{\mathbb{R}^n}
						u\cdot\rho(u-v)dx.
					\end{equation} 
					Multiplying $\eqref{Main1}_3$ by $v$ and integrating it in $\mathbb{R}^n$, we obtain
					\begin{equation}\label{T22}
						\frac{1}{2}\frac{d}{dt}\int_{\mathbb{R}^n}|v|^2dx+\|\nabla v\|_{L^2}^2=\int_{\mathbb{R}^n}
						v\cdot\rho(u-v)dx.
					\end{equation} 
					Adding \eqref{T22} into \eqref{T21}  gives
					\begin{equation}\label{T2}
						\frac{1}{2}\frac{d}{dt}\Big(\int_{\mathbb{R}^n}\rho|u|^2dx+\|v\|_{L^2}^2\Big)dx+\int_{\mathbb{R}^n}\rho|u-v|^2dx+\|\nabla v\|_{L^2}^2=0.
					\end{equation} 
					For notation convenience, we define 
					\begin{equation}\label{E0}
						E(t)\triangleq \int_{\mathbb{R}^n} \rho|u|^2dx+\|v\|_{L^2}^2,~~E_0\triangleq \int_{\mathbb{R}^n}\rho_0|u_0|^2dx+\|v_0\|_{L^2}^2.
					\end{equation}
					Integrating \eqref{T2} with time over $[0,t]$ yields 
					\begin{equation}\label{TY1}
						E(t)+2\int_0^t \int_{\mathbb{R}^n}
						\rho|u-v|^2dxd\tau
						\leq E_0.
					\end{equation}
					After a direct calculation, by the a priori assumption \eqref{a-priori est}, it is shown that 
					\begin{equation}
						\begin{aligned}
							\int_{\mathbb{R}^n}\rho|u|^2dx+\|v\|_{L^2}^2
							&\leq  2\int_{\mathbb{R}^n}\rho|u-v|^2dx+2\int_{\mathbb{R}^n}\rho|v|^2dx+\|v\|_{L^2}^2\\
							&\leq  2\int_{\mathbb{R}^n}\rho|u-v|^2dx+(2\|\rho\|_{L^\infty}+1)\|v\|_{L^2}^2\\
							&\leq  2\int_{\mathbb{R}^n}\rho|u-v|^2dx+(2\delta+1)\|v\|_{L^2}^2\\
							&\leq  2(\int_{\mathbb{R}^n}\rho|u-v|^2dx+\|v\|_{L^2}^2),
						\end{aligned}
					\end{equation}
					where we use the fact \eqref{L1} and the smallness of  $\delta$. As a result, one has
					\begin{equation}\label{T6}
						\Big(\int_{\mathbb{R}^n}\rho|u-v|^2dx+\|v\|_{L^2}^2\Big)	\geq \frac{1}{2}\Big(\int_{\mathbb{R}^n}\rho |u|^2dx+\|v\|_{L^2}^2\Big).
					\end{equation}
					To obtain the time decay rates of $\|v\|_{L^2}$, we define the sets $X_1(t)$ and  $X_1^c(t)$ as follows
					\begin{equation}\label{St}
						X_1(t)=\Big\{\xi: |\xi|\leq \Big(\frac{n}{t+n}\Big)^{\frac{1}{2}}\Big\},~X_1^c(t)=\Big\{\xi: |\xi|> \Big(\frac{n}{t+n}\Big)^{\frac{1}{2}}\Big\}.
					\end{equation}
					In terms of the Plancherel Theorem, it is easy to check that 
					$$
					\int_{\mathbb{R}^n}|\nabla v(t,x)|^2dx =\int_{\mathbb{R}^n}|\xi|^2|\hat{v}(t,\xi)|^2d\xi.
					$$
					Then we are able to verify 
					\begin{equation}\label{O7}
						\begin{aligned}
							&\frac{1}{2}\frac{d}{dt}\Big(\int_{\mathbb{R}^n} \rho|u|^2dx+\|v\|_{L^2}^2\Big)
							+\int_{\mathbb{R}^n}\rho|u-v|^2dx\\
							&+\int_{X_1(t)}|\xi|^2|\hat{v}(t,\xi)|^2d\xi+\int_{X_1^c(t)}|\xi|^2|\hat{v}(t,\xi)|^2d\xi = 0.
						\end{aligned}
					\end{equation}
					When $\xi \in X_1^c(t)$, it holds
					$$
					\int_{X_1^c(t)}|\xi|^2|\hat{v}(t,\xi)|^2d\xi\geq \frac{n}{t+n} \int_{\mathbb{R}^n}|\hat{v}(t,\xi)|^2d\xi.
					$$
					Then, \eqref{O7} can be formulated as follows
					\begin{equation}\label{O8}
						\frac{1}{2}\frac{d}{dt}\Big(\int_{\mathbb{R}^n} \rho|u|^2dx+\|v\|_{L^2}^2\Big)+\int_{\mathbb{R}^n}\rho|u-v|^2dx
						+\frac{n}{t+n}\int_{X_1^c(t)}|\hat{v}(t,\xi)|^2d\xi \leq 0.
					\end{equation}
					We add a term $\frac{n}{t+n}\int_{X_1(t)}|\hat{v}(t,\xi)|^2d\xi$ on \eqref{O8} both two sides to prove 
					$$
					\begin{aligned}
						&\frac{1}{2}\frac{d}{dt}\Big(\int_{\mathbb{R}^n} \rho|u|^2dx+\|v\|_{L^2}^2\Big)
						+\int_{\mathbb{R}^n}\rho|u-v|^2dx+\frac{n}{t+n}\int_{\mathbb{R}^n}|\hat{v}(t,\xi)|^2d\xi \\
						&\leq \frac{n}{t+n}\int_{X_1(t)}|\hat{v}(t,\xi)|^2d\xi \\
						&\leq n(1+t)^{-1}\int_{X_1(t)}|\hat{v}(t,\xi)|^2d\xi .
					\end{aligned}
					$$
					Noting that $0<\frac{n}{t+n}<1$ and using Plancherel Theorem again, we have 
					$$
					\begin{aligned}
						&\frac{1}{2}\frac{d}{dt}\Big(\int_{\mathbb{R}^n} \rho|u|^2dx+\|v\|_{L^2}^2\Big)
						+\frac{n}{t+n}\Big(\int_{\mathbb{R}^n}\rho|u-v|^2dx+\|v\|_{L^2}^2\Big)\\
						&\leq n(1+t)^{-1}\int_{X_1(t)}|\hat{v}(t,\xi)|^2d\xi,
					\end{aligned}
					$$
					which, together with \eqref{T6} and the definition of $E(t)$, yields 
					\begin{equation}\label{T72}
						\frac{d}{dt}E(t)+\frac{n}{t+n}E(t)\leq 2n(1+t)^{-1}\int_{X_1(t)}|\hat{v}(t,\xi)|^2d\xi.
					\end{equation}
					The next goal is to establish the estimate of $\hat{v}(t,\xi)$. First, we rewrite $\eqref{Main1}_3$ as follows
					\begin{equation}\label{T7}
						v_t-\Delta v=F,\quad \text{and}\quad F\triangleq-v\cdot \nabla v-\nabla P+\rho(u-v).
					\end{equation}
					Taking Fourier transform, we have by Duhamel's principle that 
					\begin{equation}\label{T71}
						\hat{v}(t,\xi)=e^{-|\xi|^2t} \hat{v}_0+\int_0^te^{-|\xi|^2(t-\tau)}\hat{F}(\tau,\xi)d\tau. 
					\end{equation}
					Taking the operator ${\rm{div}}$ on $\eqref{Main1}_3$ shows
					$$
					-\Delta P={\rm{div}} (v\cdot\nabla v-\rho(u-v)),
					$$
					which implies
					$$
					\nabla P=	\nabla(-\Delta)^{-1}{\rm{div}} (v\cdot\nabla v-\rho(u-v))=\nabla(-\Delta)^{-1}{\rm{div}}{\rm{div}}(v\otimes v)
					-\nabla(-\Delta)^{-1}{\rm{div}}(\rho(u-v)).
					$$
					Therefore, the Fourier transform of $F$ can be expressed as
					\begin{equation}\label{L2}
						|\hat{F}(\tau,\xi)|\leq (2\pi)^{-\frac{n}{2}} |\xi|\|v\|_{L^2}^2+(2\pi)^{-\frac{n}{2}}\|\rho(u-v)\|_{L^1}\leq
						|\xi|\|v\|_{L^2}^2+\|\rho(u-v)\|_{L^1}.
					\end{equation}
					We define the energy functional $M(t)$ as follows 
					\begin{equation}\label{MTT}
						M(t)=\sup_{0\leq \tau\leq t} \Big\{
						(1+\tau)^{\frac{n}{2}}\big(\int_{\mathbb{R}^n}\rho|u|^2dx
						+\|v(\tau)\|_{L^2}^2\big)
						\Big\}	.
					\end{equation}
					It is easy to verify 
					\begin{equation}\label{ETV}
						E(t)=\int_{\mathbb{R}^n}\rho|u|^2dx+\|v(t)\|_{L^2}^2 \leq (1+t)^{-\frac{n}{2}} M(t).
					\end{equation}
					By \eqref{T2} and the definition of $E(t)$, we get
					\begin{equation}\label{T8}
						\frac{d}{dt} E(t)+\int_{\mathbb{R}^n}\rho|u-v|^2dx\leq 0.
					\end{equation}
					Choosing $t_0$ be a fixed positive constant determined later. For $t>t_0$, we multiply $\eqref{T8}$ by $(1+t)^{\frac{n+2}{4}}$ to prove 
					\begin{equation}\label{T9}
						\begin{aligned}
							&\frac{d}{dt}[(1+t)^{\frac{n+2}{4}} E(t)]+(1+t)^{\frac{n+2}{4}}\int_{\mathbb{R}^n}\rho|u-v|^2dx \\
							&\leq \frac{n+2}{4}(1+t)^{\frac{n-2}{4}}E(t)\leq \frac{n+2}{4} (1+t)^{-\frac{n+2}{4}}M(t) \\
							&\leq \frac{n+2}{4}(1+t_0)^{-\frac{1}{8}}(1+t)^{-\frac{n}{4}-\frac{3}{8}}M(t).
						\end{aligned}
					\end{equation}
					Define
					\begin{equation}\label{t0}
						\varepsilon_0\triangleq (1+t_0)^{-\frac{1}{4}},\quad 	\varepsilon_1\triangleq (1+t_0)^{\frac{n+2}{4}}.
					\end{equation}
					Integrating \eqref{T9} with time over $[t_0,t]$, then using $n\geq3$ and \eqref{TY1}, we have 
					\begin{equation}\label{ET}
						\begin{aligned}
							&(1+t)^{\frac{n+2}{4}} E(t)+\int_{t_0}^t (1+\tau)^{\frac{n+2}{4}}\int_{\mathbb{R}^n}\rho|u-v|^2dxd\tau \\
							&\leq \frac{n+2}{4}(1+t_0)^{-\frac{1}{8}} M(t)\int_{t_0}^t(1+\tau)^{-\frac{n}{4}-\frac{3}{8}}d\tau+(1+t_0)^{\frac{n+2}{4}}E(t_0)\\
							&\leq\frac{2n+4}{2n-5} \varepsilon_0 M(t)+(1+t_0)^{\frac{n+2}{4}}E_0\\
							&\leq 10\varepsilon_0 M(t)+\varepsilon_1E_0,
						\end{aligned}
					\end{equation}
					and 
					\begin{equation}
						\int_0^{t_0} (1+\tau)^{\frac{n+2}{4}}\int_{\mathbb{R}^n}\rho|u-v|^2dxd\tau\leq (1+t_0)^{\frac{n+2}{4}}\int_0^{t} \int_{\mathbb{R}^n}\rho|u-v|^2dxd\tau\leq  (1+t_0)^{\frac{n+2}{4}} E_0\leq \varepsilon_1 E_0.
					\end{equation}
					Therefore, we have 
					\begin{equation}
						\int_0^t (1+\tau)^{\frac{n+2}{4}}\int_{\mathbb{R}^n}\rho|u-v|^2dxd\tau\leq 
						10\varepsilon_0 M(t)+2\varepsilon_1 E_0.
					\end{equation}
					It follows from H$\ddot{\text{o}}$lder inequality to prove 
					\begin{equation}\label{phi1}
						\begin{aligned}
							\int_0^t\|\rho(u-v)(\tau)\|_{L^1}d\tau 
							&\leq \int_0^t\big \|\rho\|_{L^1}^{\frac{1}{2}}(\int_{\mathbb{R}^n}\rho|u-v|^2dx\big)^{\frac{1}{2}} d\tau\\
							&\leq \|\rho_0\|_{L^1}^{\frac{1}{2}}\int_0^t(1+\tau)^{-\frac{n+2}{8}}(1+\tau)^{\frac{n+2}{8}}\big(\int_{\mathbb{R}^n}\rho|u-v|^2dx\big)^{\frac{1}{2}} d\tau\\
							&\leq \|\rho_0\|_{L^1}^{\frac{1}{2}}\big(\int_0^t(1+\tau)^{-\frac{n+2}{4}}d\tau\big)^{\frac{1}{2}}
							\big(\int_0^t(1+\tau)^{\frac{n+2}{4}}\int_{\mathbb{R}^n}\rho|u-v|^2dxd\tau\big)^{\frac{1}{2}} \\	
							&\leq 2\|\rho_0\|_{L^1}^{\frac{1}{2}}(10\varepsilon_0 M(t)+2\varepsilon_1E_0)^{\frac{1}{2}},
						\end{aligned}
					\end{equation}
					where we use the fact that $ \|\rho\|_{L^1}=\|\rho_0\|_{L^1}$. By \eqref{T71}, we have 
					\begin{equation}
						\begin{aligned}
							|\hat{v}(t,\xi)|&\leq e^{-|\xi|^2t}	|\hat{v}_0(\xi)|+\int_0^te^{-|\xi|^2(t-\tau)}|\hat{F}(\tau,\xi)|d\tau\\
							&\leq (2\pi)^{-\frac{n}{2}}e^{-|\xi|^2t}	\|v_0\|_{L^1}+\int_0^te^{-|\xi|^2(t-\tau)}\big(|\xi|\|v\|_{L^2}^2+\|\rho(u-v)\|_{L^1}\big)d\tau\\
							&\leq  (2\pi)^{-\frac{n}{2}}e^{-|\xi|^2t}	\|v_0\|_{L^1}+|\xi|M(t)\int_0^te^{-|\xi|^2(t-\tau)}(1+\tau)^{-\frac{n}{2}}d\tau
							+\int_0^t\|\rho(u-v)\|_{L^1}d\tau\\
							&\leq  (2\pi)^{-\frac{n}{2}}e^{-|\xi|^2t}	\|v_0\|_{L^1}+2|\xi| e^{-\frac{1}{2}|\xi|^2t}M(t)\\
							&~~+2^{\frac{n}{2}}|\xi|^{-1}(1+t)^{-\frac{n}{2}}M(t)+
							2\|\rho_0\|_{L^1}^{\frac{1}{2}}(10\varepsilon_0 M(t)+2\varepsilon_1E_0)^{\frac{1}{2}}.		
						\end{aligned}	
					\end{equation}
					Then, it holds 
					\begin{equation}\label{Main10}
						\begin{aligned}
							&\int_{X_1(t)}|\hat{v}(t,\xi)|^2d\xi\\
							&\leq  C_1(1+t)^{-\frac{n}{2}}\big(\|v_0\|_{L^1}^2+\|\rho_0\|_{L^1}(10\varepsilon_0 M(t)+2\varepsilon_1E_0)\big)+C_2(1+t)^{-\frac{n+2}{2}}M(t)^2,
						\end{aligned}
					\end{equation}
					where the positive constants $C_1$ and $C_2$ are independent of time only depend on $n$.  
					
					The combination of \eqref{T72} and \eqref{Main10}  yields 
					\begin{equation}\label{Main11}
						\begin{aligned}
							&\frac{d}{dt}E(t)
							+\frac{n}{t+n}E(t)\\
							&\leq 2n C_1(1+t)^{-\frac{n+2}{2}}\big(\|v_0\|_{L^1}^2+\|\rho_0\|_{L^1}(10\varepsilon_0 M(t)+2\varepsilon_1E_0)\big)+2nC_2(1+t)^{-\frac{n+4}{2}}M(t)^2.
						\end{aligned}			
					\end{equation}
					Multiplying \eqref{Main11}  by $(t+n)^n$, integrating the resulting equation with time  and using the definition of $E(t)$ give
					\begin{equation}\label{Main16}
						E(t)\leq C_3(1+t)^{-\frac{n}{2}}\big(\|v_0\|_{L^1}^2+\|\rho_0\|_{L^1}(10\varepsilon_0 M(t)+2\varepsilon_1E_0)\big)+C_4(1+t)^{-\frac{n+2}{2}}M(t)^2,
					\end{equation}
					where $C_3$ and $C_4$ are positive constants independent of time and only depend on $n$.  
					
					Denote
					\begin{equation}\label{t0c}
						t_0=\max\Big\{(40C_3\|\rho_0\|_{L^1})^{4},  16C_3C_4(\|v_0\|_{L^1}^2+\|\rho_0\|_{L^1}^2)\Big\}.
					\end{equation}
					Since $\delta_0$ is sufficiently small, it holds
					$$
					E_0=\int_{\mathbb{R}^n}\rho_0|u_0|^2dx+\|v_0\|_{L^2}^2\leq C\delta_0^2\leq \frac{1}{2}\|\rho_0\|_{L^1}(1+t_0)^{-\frac{n+2}{4}}.
					$$
					Then, according to the definition of $\varepsilon_0$ and $\varepsilon_1$ in \eqref{t0}, we obtain 
					\begin{equation}\label{31702}
						10\varepsilon_0C_3 \|\rho_0\|_{L^1}\leq \frac{1}{4},\quad  4C_3C_4(\|v_0\|_{L^1}^2+\|\rho_0\|_{L^1}^2)(1+t_0)^{-1}\leq \frac{1}{4},\quad 2\varepsilon_1E_0\leq \|\rho_0\|_{L^1}.
					\end{equation}
					Due to the smallness of $M(0)$ and the continuity of $M(t)$, we give the a priori assumption that
					\begin{equation}\label{31701}
						M(t)\leq 4C_3(\|v_0\|_{L^1}^2+\|\rho_0\|_{L^1}^2).
					\end{equation}
					By the definition of $M(t)$, \eqref{Main16} and \eqref{31702}, then it holds 
					\begin{equation}\label{T74}
						\begin{aligned}
							M(t)
							&\leq C_3\big(\|v_0\|_{L^1}^2+\|\rho_0\|_{L^1}(10\varepsilon_0 M(t)+2\varepsilon_1E_0)\big)+C_4(1+t)^{-1}M(t)^2\\
							&\leq C_3\big(\|v_0\|_{L^1}^2+\|\rho_0\|_{L^1}(10\varepsilon_0 M(t)+2\varepsilon_1E_0)\big)+C_4(1+t_0)^{-1}M(t)^2\\
							&\leq \frac{1}{2}M(t)+C_3(\|v_0\|_{L^1}^2+\|\rho_0\|_{L^1}^2),
						\end{aligned}
					\end{equation}
					which implies 
					\begin{equation}\label{Mt}
						M(t)\leq 2C_3(\|v_0\|_{L^1}^2+\|\rho_0\|_{L^1}^2).
					\end{equation}
					Hence we close the a priori assumption \eqref{31701} and obtain 
					\begin{equation}
						M(t)\leq C(\|v_0\|_{L^1}^2+\|\rho_0\|_{L^1}^2).
					\end{equation}
					Consequently, it follows from \eqref{ETV} to prove for $t>t_0$ that
					\begin{equation}\label{O9}
						\Big(\int_{\mathbb{R}^n}\rho|u|^2dx\Big)^{\frac{1}{2}}\leq C(\|v_0\|_{L^1}+\|\rho_0\|_{L^1})(1+t)^{-\frac{n}{4}},\quad 
						\|v(t)\|_{L^2}\leq C(\|v_0\|_{L^1}+\|\rho_0\|_{L^1})(1+t)^{-\frac{n}{4}}.
					\end{equation}
					It is shown from  \eqref{t0c} that $t_0$ is a bounded positive constant. For  $0\leq t\leq t_0$, we have 
					\begin{equation}\label{O10}
						E(t)=\int_{\mathbb{R}^n} \rho|u|^2dx+\|v\|_{L^2}^2\leq E_0 \leq CE_0(1+t)^{-\frac{n}{2}}\leq C\mathcal{I}^{2}_0(1+t)^{-\frac{n}{2}}.
					\end{equation}
					The combination of \eqref{O9} and \eqref{O10} yields the following decay estimates for $t\geq 0$
					\begin{equation}\label{O11}
						\Big(\int_{\mathbb{R}^n}\rho|u|^2dx\Big)^{\frac{1}{2}}\leq C\mathcal{I}_0(1+t)^{-\frac{n}{4}},\quad 
						\|v(t)\|_{L^2}\leq C\mathcal{I}_0(1+t)^{-\frac{n}{4}}.
					\end{equation}
					
					Now we are in a position to prove the time decay rate of $u(t,x)$. Since we assume $\rho(t,x)>0$, the second equation in $\eqref{Main1}$ can be formulated to 
					\begin{equation}\label{newu}
						u_t+u\cdot\nabla u+u-v=0.	
					\end{equation}
					Multiplying \eqref{newu} by $u$ and integrating it  in $\mathbb{R}^n$ yields 
					$$
					\frac{1}{2}\frac{d}{dt}\|u\|_{L^2}^2+\|u\|_{L^2}^2=\int_{\mathbb{R}^n}(-u\cdot\nabla u+v)\cdot udx.
					$$
					It is proved by \eqref{O11}  that 
					$$
					\begin{aligned}
						&\Big| \int_{\mathbb{R}^n}(-u\cdot\nabla u+v)\cdot udx\Big|\\
						&\leq \|u\|_{L^2}\|\nabla u\|_{L^\infty}\|u\|_{L^2}+\|v\|_{L^2}\|u\|_{L^2}\\
						&\leq C\|\nabla^2u\|_{H^{s-2}}\|u\|_{L^2}^2+\frac{1}{4}\|u\|_{L^2}^2+\|v\|_{L^2}^2\\
						&\leq C\delta\|u\|_{L^2}^2+\frac{1}{4}\|u\|_{L^2}^2+\|v\|_{L^2}^2\\
						&\leq\frac{1}{2}\|u\|_{L^2}^2+C\mathcal{I}_0^2(1+t)^{-\frac{n}{2}},
					\end{aligned}
					$$
					where we have used the fact that since $s\geq 2[\frac{n}{2}]+1$, using \eqref{L1} yields
					\begin{equation}
					\|\nabla u\|_{L^\infty}\leq C\|\nabla^2u\|_{H^{[\frac{n}{2}]}}\leq C\|\nabla^2u\|_{H^{s-2}}.
					\end{equation}
					Therefore, we have 
					$$
					\frac{1}{2}\frac{d}{dt}\|u\|_{L^2}^2+\frac{1}{2}\|u\|_{L^2}^2\leq C\mathcal{I}_0^2(1+t)^{-\frac{n}{2}}.
					$$
					By Gr$\ddot{\text{o}}$nwall  inequality, it holds for $t\geq 0$ that 
					$$
					\|u(t)\|_{L^2}\leq C\mathcal{I}_0(1+t)^{-\frac{n}{4}}.
					$$
					This completes the proof of this proposition.
				\end{proof}
				
				Now we are going to establish the time decay estimates of high-order derivatives of the velocities.
				\begin{prop}\label{p3}
					Assume the conditions in Theorem \ref{Th1} hold. Let $(\rho,u,v)(t,x)$ be the classical solution of the system \eqref{Main1}-\eqref{ID} satisfying the a priori assumption \eqref{a-priori est}. Then, it holds for $t\geq 0$ that 
					\begin{equation}\label{high}
						\|\nabla u(t)\|_{H^s}\leq C\mathcal{I}_0(1+t)^{-\frac{n}{4}},\quad 
						\|\nabla v(t)\|_{H^s}\leq C\mathcal{I}_0(1+t)^{-\frac{n}{4}},	
					\end{equation}
					where the positive constant $C$ is independent of time.
				\end{prop}
				\begin{proof}
					To deal with the high-order derivatives, we first give the assumption of the energy functional  $N(t)$ as follows
					\begin{equation}
						N(t)=\sup_{0\leq \tau\leq t} \Big\{ (1+\tau)^{\frac{n}{4}}(\|\nabla u(\tau)\|_{H^s}+\|\nabla v(\tau)\|_{H^s}) \Big\},	
					\end{equation}
					with the integer $s\geq 2[\frac{n}{2}]+1$. It is easy to check  that 
					\begin{equation}
						\|\nabla u(t)\|_{H^s}\leq (1+t)^{-\frac{n}{4}}N(t),\quad 	\|\nabla v(t)\|_{H^s}\leq (1+t)^{-\frac{n}{4}}N(t).
					\end{equation}
					For $1\leq k \leq s+1$, applying $\nabla^k$ to $\eqref{Main1}_3$ and taking the $L^2$ inner product with $\nabla^k v$ in $\mathbb{R}^n$ yield 
					\begin{equation}\label{H1}
						\frac{1}{2}\frac{d}{dt}\|\nabla^k v\|_{L^2}^2+\|\nabla \nabla^kv\|_{L^2}^2=-\int_{\mathbb{R}^n}\nabla^k(v\cdot\nabla v)\cdot\nabla^kvdx
						+\int_{\mathbb{R}^n}\nabla^k(\rho(u-v))\cdot\nabla^k vdx.
					\end{equation}
					In terms of Lemma \ref{lem1}, the first term on the right-hand side of \eqref{H1} can be estimated below
					$$
					\begin{aligned}
						&\Big| \int_{\mathbb{R}^n}\nabla^k(v\cdot\nabla v)\cdot\nabla^kvdx\Big|\\
						&\leq \Big| \int_{\mathbb{R}^n}[\nabla^k(v\cdot\nabla v)-v\nabla^k\nabla v]\cdot\nabla^kvdx\Big|
						+\Big| \int_{\mathbb{R}^n}v\nabla^k\nabla v\cdot\nabla^kvdx\Big|\\
						&\leq C\|\nabla v\|_{L^\infty}\|\nabla^kv\|_{L^2}^2\\
						&\leq C\|\nabla^2v\|_{H^{s-2}}\|\nabla^kv\|_{L^2}^2\\
						&\leq C\delta(1+t)^{-\frac{n}{2}}N(t)^2,
					\end{aligned}
					$$
					here we use the a priori assumption \eqref{a-priori est} and the fact that 
					$$
					\int_{\mathbb{R}^n}v\nabla^k\nabla v\cdot\nabla^kvdx=-\frac{1}{2}\int_{\mathbb{R}^n}\text{div}v|\nabla^kv|^2dx=0,~~~ \|\nabla^kv\|_{L^2}\leq (1+t)^{-\frac{n}{4}}N(t) ~~\text{for}~~1\leq k\leq s+1.
					$$
					The second term on the right-hand side of \eqref{H1} is stated as 
					$$
					\begin{aligned}
						&\Big| \int_{\mathbb{R}^n}\nabla^k(\rho(u-v))\cdot\nabla^kvdx\Big|\\
						&\leq \Big| \int_{\mathbb{R}^n}[\nabla^k(\rho(u-v))-(u-v)\cdot\nabla^k\rho]\cdot\nabla^kvdx\Big|
						+\Big| \int_{\mathbb{R}^n}(u-v)\cdot\nabla^k\rho \cdot\nabla^kvdx\Big|\\
						&\leq \Big| \int_{\mathbb{R}^n}[\nabla^k(\rho(u-v))-(u-v)\cdot\nabla^k\rho]\cdot\nabla^kvdx\Big|
						+\Big| \int_{\mathbb{R}^n}\nabla (u-v)\cdot\nabla^{k-1}\rho \cdot\nabla^kvdx\Big|+\Big| \int_{\mathbb{R}^n} (u-v)\cdot\nabla^{k-1}\rho \cdot\nabla\nabla^{k}vdx\Big|\\
						&\leq C(\|\nabla(u-v)\|_{L^\infty}\|\nabla^{k-1}\rho\|_{L^2}+
						\|\nabla^k(u-v)\|_{L^2}\|\rho\|_{L^\infty})\|\nabla^kv\|_{L^2}\\
						&\quad+C\|\nabla(u-v)\|_{L^\infty}\|\nabla^{k-1}\rho\|_{L^2}\|\nabla^kv\|_{L^2}+
						\|u-v\|_{L^\infty}\|\nabla^{k-1}\rho\|_{L^2}\|\nabla\nabla^{k}v\|_{L^2}\\
						&\leq C(\|\nabla^2(u-v)\|_{H^{s-2}}\|\nabla^{k-1}\rho\|_{L^2}+
						\|\nabla^k(u-v)\|_{L^2}\|\nabla\rho\|_{H^{s-2}})\|\nabla^kv\|_{L^2}\\
						&\quad+C\|\nabla^2(u-v)\|_{H^{s-2}}\|\nabla^{k-1}\rho\|_{L^2}\|\nabla^kv\|_{L^2}+C\|\nabla(u-v)\|_{H^{s-2}}\|\nabla^{k-1}\rho\|_{L^2}\|\nabla\nabla^{k}v\|_{L^2}\\
						&\leq C\delta(\|\nabla u\|_{H^{s}}+\|\nabla v\|_{H^{s}})(\|\nabla^kv\|_{L^2}+\|\nabla\nabla^kv\|_{L^2})\\
						&\leq C\delta(1+t)^{-\frac{n}{2}}N(t)^2+\frac{1}{2}\|\nabla\nabla^kv\|_{L^2}^2.
					\end{aligned}
					$$
					It then concludes from the above estimates to obtain for $1\leq k\leq s+1$ that  
					\begin{equation}\label{H2}
						\frac{d}{dt}\|\nabla^k v\|_{L^2}^2+\|\nabla \nabla^kv\|_{L^2}^2
						\leq C\delta(1+t)^{-\frac{n}{2}}N(t)^2.
					\end{equation}
					Define the sets $X_2(t)$ and  $X_2^c(t)$ as follows
					\begin{equation}\label{x2t}
						X_2(t)=\Big\{\xi: |\xi|\leq 1\Big\},~X_2^c(t)=\Big\{\xi: |\xi|>1\Big\}.
					\end{equation}
					With the help of Plancherel Theorem, it holds 
					\begin{equation}\label{H3}
						\begin{aligned}
							&\frac{d}{dt}\|\nabla^kv\|_{L^2}^2
							+\int_{X_2(t)}|\xi|^{2k+2}|\hat{v}(t,\xi)|^2d\xi+\int_{X_2^c(t)}|\xi|^{2k+2}|\hat{v}(t,\xi)|^2d\xi \\
							&\leq C\delta(1+t)^{-\frac{n}{2}}N(t)^2.
						\end{aligned}
					\end{equation}
					Noting that if $\xi \in X_2^c(t)$, we have 
					$$
					\int_{X_2^c(t)}|\xi|^{2k+2}|\hat{v}(t,\xi)|^2d\xi\ge \int_{X_2^c(t)}|\xi|^{2k}|\hat{v}(t,\xi)|^2d\xi.
					$$
					Therefore, \eqref{H3} can be written into 
					\begin{equation}\label{H4}
						\frac{d}{dt}\|\nabla^kv\|_{L^2}^2+
						\int_{X_2^c(t)}|\xi|^{2k}|\hat{v}(t,\xi)|^2d\xi \leq C\delta(1+t)^{-\frac{n}{2}}N(t)^2.
					\end{equation}
					We add a term $\int_{X_2(t)}|\xi|^{2k}|\hat{v}(t,\xi)|^2d\xi$ on \eqref{H4} both two sides to prove 
					$$
					\begin{aligned}
						&\frac{d}{dt}\|\nabla^kv\|_{L^2}^2
						+\int_{\mathbb{R}^n}|\xi|^{2k}|\hat{v}(t,\xi)|^2d\xi \\
						&\leq \int_{X_2(t)}|\xi|^{2k}|\hat{v}(t,\xi)|^2d\xi+C\delta(1+t)^{-\frac{n}{2}}N(t)^2.
					\end{aligned}
					$$
					Using Plancherel Theorem again 	and Proposition \ref{p2} yield
					\begin{equation}\label{H6}
						\begin{aligned}
							\frac{d}{dt}\|\nabla^kv\|_{L^2}^2
							+\|\nabla^kv\|_{L^2}^2
							&	\leq \int_{X_2(t)}|\xi|^{2k}|\hat{v}(t,\xi)|^2d\xi+C\delta(1+t)^{-\frac{n}{2}}N(t)^2\\
							&	\leq \int_{X_2(t)}|\hat{v}(t,\xi)|^2d\xi+C\delta(1+t)^{-\frac{n}{2}}N(t)^2\\
							& \leq 	\int_{\mathbb{R}^3}|\hat{v}(t,\xi)|^2d\xi+C\delta(1+t)^{-\frac{n}{2}}N(t)^2\\
							& \leq \|v\|_{L^2}^2+C\delta(1+t)^{-\frac{n}{2}}N(t)^2\\
							&\leq  C\mathcal{I}_0^2(1+t)^{-\frac{n}{2}}+C\delta(1+t)^{-\frac{n}{2}}N(t)^2.
						\end{aligned}
					\end{equation}
					In terms of Gr$\ddot{\text{o}}$nwall  inequality, we are able to prove 
					$$
					\|\nabla^k v\|_{L^2}\leq C\mathcal{I}_0(1+t)^{-\frac{n}{4}}+
					C\delta^{\frac12} (1+t)^{-\frac{n}{4}}N(t).
					$$
					Summing $k$ from 1 to $s+1$ yields
					\begin{equation}\label{H8}
						\|\nabla v(t)\|_{H^{s}}\leq C(1+t)^{-\frac{n}{4}}(\mathcal{I}_0+\delta^{\frac12} N(t)).
					\end{equation}
					
					Let the integer $1\leq k\leq s+1$. Applying $\nabla^k$ to \eqref{newu}, multiplying it by $\nabla^ku$, and integrating it in $\mathbb{R}^n$ yields 
					\begin{equation}\label{H9}
						\begin{aligned}
							\frac{1}{2}\frac{d}{dt}\|\nabla^ku\|_{L^2}^2+\|\nabla^ku\|_{L^2}^2
							=\int_{\mathbb{R}^n}\nabla^k(-u\cdot\nabla u)\cdot\nabla^kudx+\int_{\mathbb{R}^n}\nabla^ku\cdot\nabla^kvdx.
						\end{aligned}	
					\end{equation}
					By H$\ddot{\text{o}}$lder inequality, Sobolev inequality and Lemma  \ref{lem1}, we can prove
					\begin{equation}\label{H10}
						\begin{aligned}
							&\Big|\int_{\mathbb{R}^n}\nabla^k(-u\cdot\nabla u)\cdot\nabla^kudx\Big|+\Big|\int_{\mathbb{R}^n}\nabla^ku\cdot\nabla^kvdx\Big|\\
							&\leq \Big| \int_{\mathbb{R}^n}(\nabla^k(u\cdot\nabla u)-u\cdot \nabla\nabla^ku)\cdot\nabla^kudx\Big|+  \Big| \int_{\mathbb{R}^n}u\cdot \nabla\nabla^ku\cdot\nabla^kudx\Big|+\Big|\int_{\mathbb{R}^n}\nabla^ku\cdot\nabla^kvdx\Big|\\
							&\leq C\|\nabla u\|_{L^\infty}\|\nabla^ku\|_{L^2}^2+\frac{1}{2}\|\text{div}u\|_{L^\infty}\|\nabla^ku\|_{L^2}^2
							+\frac{1}{4}\|\nabla^ku\|_{L^2}^2+\|\nabla^kv\|_{L^2}^2\\
							&\leq C\|\nabla^2 u\|_{H^{s-2}}\|\nabla^ku\|_{L^2}^2+\frac{1}{4}\|\nabla^ku\|_{L^2}^2+\|\nabla^kv\|_{L^2}^2\\
							&\leq C\delta\|\nabla^ku\|_{L^2}^2+\frac{1}{4}\|\nabla^ku\|_{L^2}^2+\|\nabla^kv\|_{L^2}^2
							\\
							&\leq \frac{1}{2}\|\nabla^ku\|_{L^2}^2+\|\nabla^kv\|_{L^2}^2.
						\end{aligned}	
					\end{equation}
					The combination of \eqref{H9} and \eqref{H10} and summing $k$ from 1 to $s+1$ gives 
					\begin{equation}\label{H11}
						\begin{aligned}
							\frac{1}{2}\frac{d}{dt}\|\nabla u\|_{H^{s}}^2+\frac{1}{2}\|\nabla u\|_{H^{s}}^2
							\leq \|\nabla v\|_{H^{s}}^2.	
						\end{aligned}	
					\end{equation}
					From \eqref{H8}, it holds 
					\begin{equation}\label{H12}
						\begin{aligned}
							\frac{d}{dt}\|\nabla u\|_{H^{s}}^2+\|\nabla u\|_{H^{s}}^2
							\leq C(1+t)^{-\frac{n}{2}}(\mathcal{I}_0+\delta^{\frac12} N(t))^2.
						\end{aligned}	
					\end{equation}
					In terms of  Gr$\ddot{\text{o}}$nwall inequality, we have
					\begin{equation}\label{H13}
						\|\nabla u(t)\|_{H^{s}}\leq C(1+t)^{-\frac{n}{4}}(\mathcal{I}_0+\delta^{\frac12} N(t)).
					\end{equation}
					Finally, it then follows from \eqref{H8}, \eqref{H13} and the definition of $N(t)$ to show
					$$
					N(t)\leq C\mathcal{I}_0+C\delta^{\frac12} N(t).
					$$
					Since $\delta$ is sufficiently small, we immediately obtain 
					\begin{equation}\label{31601}
						N(t)\leq C\mathcal{I}_0.
					\end{equation}
					Consequently, one has 
					$$
					\|\nabla u(t)\|_{H^s}\leq C\mathcal{I}_0(1+t)^{-\frac{n}{4}},\quad \|\nabla v(t)\|_{H^s}\leq C\mathcal{I}_0(1+t)^{-\frac{n}{4}}.
					$$
					This completes the proof of the proposition.
				\end{proof}
				
				To establish the upper and lower bound of density, we give the weighted time decay estimate of $u$ in the next proposition.
				
				\begin{prop}\label{p6}
					Assume the conditions in Theorem \ref{Th1} hold. Let $(\rho,u,v)(t,x)$ be the classical solution of the system \eqref{Main1}-\eqref{ID} satisfying the a priori assumption \eqref{a-priori est}. Then, it holds for $t\geq 0$ that 
					\begin{equation}
						\int_0^t(1+\tau)^{\frac{9}{8}}\|\nabla u(\tau)\|_{H^s}^2d\tau \leq C(\mathcal{I}_0^2+\mathcal{I}_0^3),
					\end{equation}	
					where the positive constant $C$ is independet of time.
				\end{prop}
				\begin{proof}
					First, we consider the lower order derivatives of $v$. In terms of \eqref{T2}, one has 
					\begin{equation}
						\frac{d}{dt}E(t)+\|\nabla v\|_{L^2}^2\leq 0.
					\end{equation}
					Multiplying it by $(1+t)^{9/8}$ gives 
					\begin{equation}\label{31501}
						\frac{d}{dt}\Big\{(1+t)^{\frac{9}{8}}E(t)\Big\}+(1+t)^{\frac{9}{8}}\|\nabla v\|_{L^2}^2\leq \frac{9}{8}(1+t)^{\frac{1}{8}}E(t).
					\end{equation}
					It follows from Proposition \ref{p2} that
					\begin{equation}
						E(t)\leq C\mathcal{I}_0^2(1+t)^{-\frac{n}{2}},
					\end{equation}
					therefore integrating \eqref{31501} with time yields 
					\begin{equation}\label{31503}
						\int_0^t(1+\tau)^{\frac{9}{8}}\|\nabla v(\tau)\|_{L^2}^2d\tau\leq C\mathcal{I}_0^2\int_0^t(1+\tau)^{-\frac{n}{2}+\frac{1}{8}}d\tau+E_0\leq C\mathcal{I}_0^2.
					\end{equation}
					Next, we investigate the weighted time decay estimate for $\|\nabla u\|_{L^2}$. Taking $\int_{\mathbb{R}^n}\nabla \eqref{newu}\cdot\nabla udx$, it holds
					\begin{equation}
						\begin{aligned}
							\frac{1}{2}\frac{d}{dt}\|\nabla u\|_{L^2}^2+\|\nabla u\|_{L^2}^2
							&\leq \Big|\int_{\mathbb{R}^n}\nabla(-u\cdot\nabla u)\cdot\nabla udx\Big|+\Big|\int_{\mathbb{R}^n}\nabla u\cdot\nabla vdx\Big|\\
							&\leq  C\mathcal{I}_0^3(1+t)^{-\frac{3n}{4}}+\frac{1}{2}\|\nabla u\|_{L^2}^2+\frac{1}{2}\|\nabla v\|_{L^2}^2.	
						\end{aligned}	
					\end{equation}
					Then  
					\begin{equation}\label{31502}
						\frac{d}{dt}\|\nabla u\|_{L^2}^2+\|\nabla u\|_{L^2}^2
						\leq  C\mathcal{I}_0^3(1+t)^{-\frac{3n}{4}}+\|\nabla v\|_{L^2}^2.	
					\end{equation}
					Multiplying \eqref{31502} by $(1+t)^{9/8}$, intergrating it with time and using \eqref{31503}, one has   
					\begin{equation}\label{31505}
						\begin{aligned}
							\int_0^t(1+\tau)^{\frac{9}{8}}\|\nabla u(\tau)\|_{L^2}^2d\tau
							&\leq C\mathcal{I}_0^3\int_0^t(1+\tau)^{-\frac{3n}{4}+\frac{9}{8}}d\tau+\int_0^t(1+\tau)^{\frac{9}{8}}\|\nabla v(\tau)\|_{L^2}^2d\tau+\|\nabla u_0\|_{L^2}^2\\
							&\leq C(\mathcal{I}_0^2+\mathcal{I}_0^3).
						\end{aligned}
					\end{equation}
					
					To obtain the weighted time decay estimates for the high-order derivatives of $u$,	we need to establish the estimates for $(u-v)$.
					
					Let $\eqref{newu}-\eqref{Main1}_3$, we obtain a new system for $(u-v)$ 
					\begin{equation}\label{u-v2}
						(u-v)_t+u-v=-u\cdot\nabla u+v\cdot\nabla v+\nabla P-\Delta v-\rho(u-v).	
					\end{equation}
					Multiply \eqref{u-v2} by $(u-v)$ and integrate over $\mathbb{R}^n$
					\begin{equation}
						\begin{aligned}
							\frac{1}{2}\frac{d}{dt}\|u-v\|_{L^2}^2+\|u-v\|_{L^2}^2=\int_{\mathbb{R}^n}
							(-u\cdot\nabla u+v\cdot\nabla v+\nabla P-\Delta v-\rho(u-v))\cdot (u-v)dx.		
						\end{aligned}
					\end{equation}
					It is easy to verify 
					$$
					\begin{aligned}
						&\Big| \int_{\mathbb{R}^n}
						(-u\cdot\nabla u+v\cdot\nabla v+\nabla P-\Delta v-\rho(u-v))\cdot (u-v)dx\Big|\\
						&\leq \Big| \int_{\mathbb{R}^n}
						(-u\cdot\nabla u+v\cdot\nabla v+\nabla P)\cdot (u-v)dx\Big|+
						\Big| \int_{\mathbb{R}^n}
						(-\Delta v-\rho(u-v))\cdot (u-v)dx\Big|\\
						&\leq (\|u\|_{L^\infty}\|\nabla u\|_{L^2}+\|v\|_{L^\infty}\|\nabla v\|_{L^2}+\|\nabla P\|_{L^2})\|u-v\|_{L^2}+\|\nabla v\|_{L^2}\|\nabla (u-v)\|_{L^2}+\|\rho\|_{L^\infty}\|u-v\|_{L^2}^2\\
						&\leq C(\|u\|_{L^\infty}\|\nabla u\|_{L^2}+\|v\|_{L^\infty}\|\nabla v\|_{L^2})\|u-v\|_{L^2}+\|\nabla v\|_{L^2}(\|\nabla u\|_{L^2}+\|\nabla v\|_{L^2})+C\|\rho\|_{L^\infty}\|u-v\|_{L^2}^2\\
						&\leq C\mathcal{I}_0^3(1+t)^{-\frac{3n}{4}}+\frac{3}{2}\|\nabla v\|_{L^2}^2+\frac{1}{2}\|\nabla u\|_{L^2}^2+\frac{1}{2}\|u-v\|_{L^2}^2,
					\end{aligned}
					$$
					where we use the fact that 
					$$
					\|\nabla P\|_{L^2}\leq C\|v\|_{L^\infty}\|\nabla v\|_{L^2}+C\|\rho\|_{L^\infty}\|u-v\|_{L^2}.
					$$
					Then, it holds 
					\begin{equation}\label{31504}
						\begin{aligned}
							\frac{d}{dt}\|u-v\|_{L^2}^2+\|u-v\|_{L^2}^2\leq C\mathcal{I}_0^3(1+t)^{-\frac{3n}{4}}+3\|\nabla v\|_{L^2}^2+\|\nabla u\|_{L^2}^2.	
						\end{aligned}
					\end{equation}
					Multiplying \eqref{31504} by $(1+t)^{9/8}$, then integrating with time and using \eqref{31503}, \eqref{31505}, one has
					\begin{equation}\label{31508}
						\int_0^t	(1+\tau)^{\frac{9}{8}}\|(u-v)(\tau)\|_{L^2}^2d\tau\leq C(\mathcal{I}_0^2+\mathcal{I}_0^3) .
					\end{equation} 
					
					Now we are in a position to estimate $\int_0^t (1+\tau)^{\frac{9}{8}}\|\nabla^2 v(\tau)\|_{L^2}^2d\tau$. Applying $\nabla $ to $\eqref{Main1}_3$ and taking the $L^2$ inner product with $\nabla v$ in $\mathbb{R}^n$ yield 
					\begin{equation}\label{31506}
						\frac{1}{2}\frac{d}{dt}\|\nabla v\|_{L^2}^2+\| \nabla^2v\|_{L^2}^2=-\int_{\mathbb{R}^n}\nabla (v\cdot\nabla v)\cdot\nabla vdx
						+\int_{\mathbb{R}^n}\nabla (\rho(u-v))\cdot\nabla vdx.
					\end{equation}
					Noting that 
					$$
					\Big|\int_{\mathbb{R}^n}\nabla (v\cdot\nabla v)\cdot\nabla vdx \Big|\leq C\mathcal{I}_0^3(1+t)^{-\frac{3n}{4}}.
					$$
					Integration by parts, we have 
					\begin{equation}
						\Big|\int_{\mathbb{R}^n}\nabla (\rho(u-v))\cdot\nabla vdx\Big|\leq C\|\rho\|_{L^\infty}\|u-v\|_{L^2}\|\nabla^2v\|_{L^2}
						\leq C\delta\|u-v\|_{L^2}^2+\frac{1}{2}\|\nabla^2v\|_{L^2}^2.
					\end{equation}
					Then, it holds
					\begin{equation}\label{31507}
						\frac{d}{dt}\|\nabla v\|_{L^2}^2+\| \nabla^2v\|_{L^2}^2\leq C\mathcal{I}_0^3(1+t)^{-\frac{3n}{4}}+C\delta\|u-v\|_{L^2}^2.
					\end{equation}
					Multiplying \eqref{31507} by $(1+t)^{9/8}$, intergrating it with time and using \eqref{31508} give
					\begin{equation}\label{31509}
						\int_0^t (1+\tau)^{\frac{9}{8}}\|\nabla^2 v(\tau)\|_{L^2}^2d\tau\leq C(\mathcal{I}_0^2+\mathcal{I}_0^3).
					\end{equation}
					Considering $\int_{\mathbb{R}^n}\nabla^2 \eqref{newu}\cdot\nabla^2 udx$, one has
					\begin{equation}
						\begin{aligned}
							\frac{1}{2}\frac{d}{dt}\|\nabla^2 u\|_{L^2}^2+\|\nabla^2 u\|_{L^2}^2
							&\leq \Big|\int_{\mathbb{R}^n}\nabla^2(-u\cdot\nabla u)\cdot\nabla^2 udx\Big|+\Big|\int_{\mathbb{R}^n}\nabla^2 u\cdot\nabla^2 vdx\Big|\\
							&\leq  C\mathcal{I}_0^3(1+t)^{-\frac{3n}{4}}+\frac{1}{2}\|\nabla ^2u\|_{L^2}^2+\frac{1}{2}\|\nabla^2 v\|_{L^2}^2.	
						\end{aligned}	
					\end{equation}
					Then 
					\begin{equation}\label{31510}
						\frac{d}{dt}\|\nabla^2 u\|_{L^2}^2+\|\nabla^2 u\|_{L^2}^2
						\leq  C\mathcal{I}_0^3(1+t)^{-\frac{3n}{4}}+\|\nabla^2 v\|_{L^2}^2.	
					\end{equation}
					We multiply \eqref{31510} by $(1+t)^{9/8}$, then intergrate over time and use \eqref{31509} to prove  
					\begin{equation}\label{31511}
						\begin{aligned}
							\int_0^t(1+\tau)^{\frac{9}{8}}\|\nabla^2 u(\tau)\|_{L^2}^2d\tau
							&\leq C\mathcal{I}_0^3\int_0^t(1+\tau)^{-\frac{3n}{4}+\frac{9}{8}}d\tau+\int_0^t(1+\tau)^{\frac{9}{8}}\|\nabla v(\tau)\|_{L^2}^2d\tau+\|\nabla^2u_0\|_{L^2}^2\\
							&\leq C(\mathcal{I}_0^2+\mathcal{I}_0^3).
						\end{aligned}
					\end{equation}
					We repeat the above process  and use induction method to prove the following weighted estimates for the high-order derivatives of $u$,
					\begin{equation}
						\int_0^t(1+\tau)^{\frac{9}{8}}\|\nabla u(\tau)\|_{H^s}^2d\tau \leq C(\mathcal{I}_0^2+\mathcal{I}_0^3).
					\end{equation}
					This completes the proof of the proposition.
				\end{proof}

				We prove the uniform-in-time upper bound and the lower bound of the time decay estimate of the fluid density $\rho(t,x)$ in the following proposition.
				\begin{prop}\label{p4}
					Assume the conditions in Theorem \ref{Th1} hold. Let $(\rho,u,v)(t,x)$ be the classical solution of the system \eqref{Main1}-\eqref{ID} satisfying the a priori assumption \eqref{a-priori est}.  Then, it holds for $t\geq 0$ that 
					\begin{equation}
						\|\rho(t)\|_{H^s}^2+	\|u(t)\|_{H^{s+2}}^2+\|v(t)\|_{H^{s+1}}^2+\int_0^t\left(\|\nabla u(\tau)\|_{H^{s+1}}^2+\|\nabla v(\tau)\|_{H^{s+1}}^2\right)d\tau
						\leq C\delta_0^2,
					\end{equation}
					and 
					\begin{equation} \label{rholower}
						\rho(t,x)\geq \rho_0(X(0,x))e^{-C(\mathcal{I}_0^2+\mathcal{I}_0^3)^{\frac{1}{2}}}>0,
					\end{equation}
					where the positive constant $C$ is independent of time.
				\end{prop}
				\begin{proof}
					Multiplying \eqref{Main1} by $\rho$ and integrating over  $\mathbb{R}^n$, we have 
					\begin{equation}\label{A11}
						\begin{aligned}
							\frac{1}{2}\frac{d}{dt}\|\rho\|_{L^2}^2&\leq \Big|\int_{\mathbb{R}^n}\rho{\rm{div}}(\rho u)dx\Big|
							\leq \Big|\int_{\mathbb{R}^n}\nabla\rho\cdot (\rho u)dx\Big|
							\leq \|\text{div}u\|_{L^\infty}\|\rho\|_{L^2}^2
							\leq C\|\nabla^2u\|_{H^{s-2}}\|\rho\|_{L^2}^2.
						\end{aligned}
					\end{equation}
					For $1\leq k\leq s$, applying $\nabla^k$ to $\eqref{Main1}_1$, then taking $L^2$ inner product with $\nabla^k\rho$ in $\mathbb{R}^n$ gives 
					$$
					\begin{aligned}
						\frac{1}{2}\frac{d}{dt}\|\nabla^k\rho\|_{L^2}^2&\leq \Big|\int_{\mathbb{R}^n}\nabla^k\rho\cdot\nabla^k{\rm{div}}(\rho u)dx\Big|\\
						&\leq  \Big|\int_{\mathbb{R}^n}\nabla^k\rho\cdot\nabla^k(\rho \text{div}u)dx\Big|+\Big|\int_{\mathbb{R}^n}\nabla^k\rho\cdot\nabla^k(u\cdot\nabla\rho)dx\Big|\\
						&\leq \Big|\int_{\mathbb{R}^n}\nabla^k\rho\cdot[\nabla^k(\rho{\rm{div}}u)-\rho\nabla^k\text{div}u]dx\Big|+\Big|\int_{\mathbb{R}^n}\nabla^k\rho\cdot\rho\nabla^k\text{div}udx\Big|\\
						&\quad+\Big|\int_{\mathbb{R}^n}\nabla^k\rho\cdot[\nabla^k(u\cdot\nabla\rho)-u\nabla^k\nabla \rho]dx\Big|+\Big|\int_{\mathbb{R}^n}\nabla^k\rho\cdot u\nabla^k\nabla \rho dx\Big|\\
						&\leq C\|\nabla^k\rho\|_{L^2}(\|\nabla\rho\|_{L^\infty}\|\nabla^k u\|_{L^2}
						+\|\text{div}u\|_{L^\infty}\|\nabla^k\rho\|_{L^2})+C\|\rho\|_{L^\infty}\|\nabla^k\rho\|_{L^2}\|\nabla^{k+1}u\|_{L^2}\\
						&\quad+C\|\nabla^k\rho\|_{L^2}(\|\nabla u\|_{L^\infty}\|\nabla^k \rho\|_{L^2}
						+\|\nabla\rho\|_{L^\infty}\|\nabla^k u\|_{L^2})+C\|\text{div}u\|_{L^\infty}\|\nabla^k \rho\|_{L^2}^2\\
						&\leq C\|\nabla^k \rho\|_{L^2}(\|\nabla^2\rho\|_{H^{s-2}}\|\nabla^k u\|_{L^2}+\|\nabla^2 u\|_{H^{s-2}}\|\nabla^k \rho\|_{L^2})
						+C\|\nabla\rho\|_{H^{s-2}}\|\nabla^k \rho\|_{L^2}\|\nabla^{k+1}u\|_{L^2}\\
						&\quad+C\|\nabla^k \rho\|_{L^2}(\|\nabla^2u\|_{H^{s-2}}\|\nabla^k \rho\|_{L^2}+\|\nabla^2\rho\|_{H^{s-2}}\|\nabla^k u\|_{L^2})+C\|\nabla^2u\|_{H^{s-2}}\|\nabla^k \rho\|_{L^2}^2.
					\end{aligned}
					$$
					Summing $k$ from $1$ to $s$, we get
					$$
					\begin{aligned}
						\frac{1}{2}\frac{d}{dt}\|\nabla \rho\|_{H^{s-1}}^2
						&\leq C\|\nabla \rho\|_{H^{s-1}}(\|\nabla^2\rho\|_{H^{s-2}}\|\nabla u\|_{H^{s-1}}+\|\nabla^2 u\|_{H^{s-2}}\|\nabla \rho\|_{H^{s-1}})
						+C\|\nabla\rho\|_{H^{s-2}}\|\nabla \rho\|_{H^{s-1}}\|\nabla^{2}u\|_{H^{s-1}}\\
						&\quad+C\|\nabla \rho\|_{H^{s-1}}(\|\nabla^2u\|_{H^{s-2}}\|\nabla \rho\|_{H^{s-1}}+\|\nabla^2\rho\|_{H^{s-2}}\|\nabla u\|_{H^{s-1}})+C\|\nabla^2u\|_{H^{s-2}}\|\nabla \rho\|_{H^{s-1}}^2\\
						&\leq C\|\nabla u\|_{H^{s}}\|\nabla \rho\|_{H^{s-1}}^2,
					\end{aligned}
					$$
					which, together with \eqref{A11}, also leads to
					\begin{equation}
						\frac{d}{dt}\|\rho\|_{H^s}^2\leq C\|\nabla u\|_{H^s}\|\rho\|_{H^s}^2.
					\end{equation}
					In terms of Proposition \ref{p6},  it holds
					\begin{equation}\label{32702}
						\begin{aligned}
							\|\rho\|_{H^s}^2&\leq \|\rho_0\|_{H^s}^2\exp\Big\{C\int_0^t\|\nabla u(\tau)\|_{H^s}d\tau\Big\}\\
							&\leq \|\rho_0\|_{H^s}^2\exp \Big\{C\int_0^t(1+\tau)^{-\frac{9}{16}}(1+\tau)^{\frac{9}{16}}\|\nabla u(\tau)\|_{H^s}d\tau\Big\}\\
							&\leq \|\rho_0\|_{H^s}^2\exp \Big\{C(\int_0^t(1+\tau)^{-\frac{9}{8}}d\tau)^{\frac{1}{2}}(\int_0^t(1+\tau)^{\frac{9}{8}}\|\nabla u(\tau)\|_{H^s}^2d\tau)^{\frac{1}{2}}\Big\}\\
							&\leq \|\rho_0\|_{H^s}^2e^{C(\mathcal{I}_0^2+\mathcal{I}_0^3)^{\frac{1}{2}}}\\
							&\leq C\delta_0^2.
						\end{aligned}
					\end{equation}
					Now we are in a position to prove the uniform bound of the $(u,v)$. Multiplying $\eqref{Main1}_3$ by $v$ and integrating the resulting equaiton in $\mathbb{R}^n$ yield
					$$
					\begin{aligned}
						&\frac{1}{2}\frac{d}{dt}\|v\|_{L^2}^2+\|\nabla v\|_{L^2}^2\\
						&=
						-\int_{\mathbb{R}^n}v\cdot\nabla v\cdot vdx
						+\int_{\mathbb{R}^n}\rho(u-v)\cdot vdx=\int_{\mathbb{R}^n}\rho(u-v)\cdot vdx.
					\end{aligned}
					$$
					By Sobolev inequality, the a priori assumption \eqref{a-priori est} and the definition of $\mathcal{D}(t)$ in \eqref{Dt}, it holds 
					$$
					\Big|\int_{\mathbb{R}^n}\rho(u-v)\cdot vdx\Big|\leq \|\rho\|_{L^n}\|u-v\|_{L^2}\|v\|_{L^{\frac{2n}{n-2}}}\leq C\delta \| u-v\|_{L^2}\|\nabla v\|_{L^2}\leq C\delta \mathcal{D}(t)^2.
					$$
					Then, we find 
					\begin{equation}\label{D3}
						\frac{1}{2}\frac{d}{dt}\|v\|_{L^2}^2+\|\nabla v\|_{L^2}^2
						\leq C\delta \mathcal{D}(t)^2.
					\end{equation}
					For the integer $1\leq k\leq s+1$, applying $\nabla^k$ to $\eqref{Main1}_3$ and multiplying it with $\nabla^k v$, then integrating in $\mathbb{R}^n$ yields 
					\begin{equation}\label{D1}
						\begin{aligned}
							&\frac{1}{2}\frac{d}{dt}\|\nabla^k v\|_{L^2}^2+\|\nabla \nabla^kv\|_{L^2}^2\\
							&=-\int_{\mathbb{R}^n}\nabla^k(v\cdot\nabla v)\cdot\nabla^kvdx
							+\int_{\mathbb{R}^n}\nabla^k(\rho(u-v))\cdot\nabla^k vdx.
						\end{aligned}
					\end{equation}
					The first term on the right-hand side of \eqref{D1} can be estimated below
					$$
					\begin{aligned}
						&\Big| \int_{\mathbb{R}^n}\nabla^k(v\cdot\nabla v)\cdot\nabla^kvdx\Big|\\
						&\leq \Big| \int_{\mathbb{R}^n}[\nabla^k(v\cdot\nabla v)-v\nabla\nabla^k v]\cdot\nabla^kvdx\Big|
						+\Big| \int_{\mathbb{R}^n}v\nabla\nabla^k v\cdot\nabla^kvdx\Big|\\
						&\leq C\|\nabla v\|_{L^\infty}\|\nabla^kv\|_{L^2}^2\\
						&\leq C\|\nabla^2v\|_{H^{s-2}}\|\nabla^kv\|_{L^2}^2\\
						&\leq  C\delta\|\nabla^kv\|_{L^2}^2.
					\end{aligned}
					$$
					The second term on the right-hand side of \eqref{D1} can be written as 
					$$
					\begin{aligned}
						&\Big| \int_{\mathbb{R}^n}\nabla^k(\rho(u-v))\cdot\nabla^kvdx\Big|\\
						&\leq \Big| \int_{\mathbb{R}^n}[\nabla^k(\rho(u-v))-(u-v)\nabla^k\rho]\cdot\nabla^kvdx\Big|
						+\Big| \int_{\mathbb{R}^n}(u-v)\cdot\nabla^k\rho \cdot\nabla^kvdx\Big|\\
						&\leq\Big| \int_{\mathbb{R}^n}[\nabla^k(\rho(u-v))-(u-v)\nabla^k\rho]\cdot\nabla^kvdx\Big|+\Big| \int_{\mathbb{R}^n}\nabla(u-v)\cdot\nabla^{k-1}\rho \cdot\nabla^kvdx\Big|
						+\Big| \int_{\mathbb{R}^n}(u-v)\cdot\nabla^{k-1}\rho \cdot\nabla\nabla^kvdx\Big|\\
						&\leq C(\|\nabla(u-v)\|_{L^\infty}\|\nabla^{k-1}\rho\|_{L^2}+
						\|\nabla^k(u-v)\|_{L^2}\|\rho\|_{L^\infty})\|\nabla^kv\|_{L^2}\\
						&\quad+C\|\nabla(u-v)\|_{L^\infty}\|\nabla^{k-1}\rho\|_{L^2}\|\nabla^kv\|_{L^2}+C\|u-v\|_{L^\infty}
						\|\nabla^{k-1}\rho\|_{L^2}\|\nabla\nabla^{k}v\|_{L^2}\\
						&\leq C(\|\nabla^2(u-v)\|_{H^{s-2}}\|\nabla^{k-1}\rho\|_{L^2}+
						\|\nabla^k(u-v)\|_{L^2}\|\nabla\rho\|_{H^{s-2}})\|\nabla^kv\|_{L^2}\\
						&\quad+C\|\nabla^2(u-v)\|_{H^{s-2}}\|\nabla^{k-1}\rho\|_{L^2}\|\nabla^kv\|_{L^2}+C\|\nabla(u-v)\|_{H^{s-2}}
						\|\nabla^{k-1}\rho\|_{L^2}\|\nabla\nabla^{k}v\|_{L^2},
					\end{aligned}
					$$	
					for $1\leq k\leq s+1$. Summing $k$ from $1$ to $s+1$ produces 
					\begin{equation}\label{D2}
						\begin{aligned}
							&\frac{1}{2}\frac{d}{dt}\|\nabla v\|_{H^s}^2+\|\nabla^2v\|_{H^s}^2\\
							&\leq C\delta \|\nabla v\|_{H^s}^2+
							C(\|\nabla^2(u-v)\|_{H^{s-2}}\|\rho\|_{H^s}+
							\|\nabla (u-v)\|_{H^{s}}\|\nabla\rho\|_{H^{s-2}})\|\nabla v\|_{H^s}\\
							&\quad+C\|\nabla^2(u-v)\|_{H^{s-2}}\|\rho\|_{H^s}\|\nabla v\|_{H^s}+C\|\nabla(u-v)\|_{H^{s-2}}
							\|\rho\|_{H^s}\|\nabla^2v\|_{H^{s}}\\
							&\leq C\delta\mathcal{D}(t)^2.
						\end{aligned}
					\end{equation}
					The combination of \eqref{D2} and \eqref{D3} yields 
					\begin{equation}\label{D4}
						\begin{aligned}
							\frac{1}{2}\frac{d}{dt}\|v\|_{H^{s+1}}^2+\|\nabla v\|_{H^{s+1}}^2
							\leq C\delta\mathcal{D}(t)^2.
						\end{aligned}
					\end{equation}
					Integrating with time over $[0,t]$ gives 
					\begin{equation}\label{TT2}
						\|v(t)\|_{H^{s+1}}^2+2\int_0^t\|\nabla v(\tau)\|_{H^{s+1}}^2d\tau\leq C\delta \int_0^t\mathcal{D}(\tau)^2d\tau+\delta_0^2.	
					\end{equation}
					Multiplying \eqref{u-v2} by $(u-v)$ and integrating over $\mathbb{R}^n$, one has
					\begin{equation}
						\begin{aligned}
							\frac{1}{2}\frac{d}{dt}\|u-v\|_{L^2}^2+\|u-v\|_{L^2}^2=\int_{\mathbb{R}^n}
							(-u\cdot\nabla u+v\cdot\nabla v+\nabla P-\Delta v-\rho(u-v))\cdot (u-v)dx.		
						\end{aligned}
					\end{equation}
					It is easy to verify 
					$$
					\begin{aligned}
						&\Big| \int_{\mathbb{R}^n}
						(-u\cdot\nabla u+v\cdot\nabla v+\nabla P-\Delta v-\rho(u-v))\cdot (u-v)dx\Big|\\
						&\leq (\|u\|_{L^\infty}\|\nabla u\|_{L^2}+\|v\|_{L^\infty}\|\nabla v\|_{L^2}+\|\nabla P\|_{L^2}+\|\Delta v\|_{L^2})\|u-v\|_{L^2}+\|\rho\|_{L^\infty}\|u-v\|_{L^2}^2\\
						&\leq C\delta \mathcal{D}(t)^2+(\frac{1}{4}+C\delta )\|u-v\|_{L^2}^2+\|\nabla^2v\|_{L^2}^2\\
						&\leq C\delta \mathcal{D}(t)^2+\frac{1}{2}\|u-v\|_{L^2}^2+\|\nabla^2v\|_{L^2}^2.
					\end{aligned}
					$$
					Hence, we obtain 
					\begin{equation}
						\begin{aligned}
							\frac{1}{2}\frac{d}{dt}\|u-v\|_{L^2}^2+\frac{1}{2}\|u-v\|_{L^2}^2\leq C\delta \mathcal{D}(t)^2+\|\nabla^2v\|_{L^2}^2.
						\end{aligned}
					\end{equation}
					Integrating with time over $[0,t]$  and using \eqref{TT2} give  
					\begin{equation}\label{TT1}
						\begin{aligned}
							&\|(u-v)(t)\|_{L^{2}}^2+\int_0^t\|(u-v)(\tau)\|_{L^{2}}^2d\tau\\
							&\leq C\delta \int_0^t\mathcal{D}(\tau)^2d\tau+2\int_0^t\|\nabla^2v(\tau)\|_{L^2}^2d\tau+\delta_0^2\\
							&\leq  C\delta \int_0^t\mathcal{D}(\tau)^2d\tau+2\delta_0^2.	
						\end{aligned}	
					\end{equation}
					Let the integer $1\leq k\leq s+2$. Applying $\nabla^k$ to \eqref{newu}, multiplying it with $\nabla^ku$, and integrating it in $\mathbb{R}^n$ yields 
					\begin{equation}\label{D9}
						\begin{aligned}
							\frac{1}{2}\frac{d}{dt}\|\nabla^ku\|_{L^2}^2+\|\nabla^ku\|_{L^2}^2
							=\int_{\mathbb{R}^n}\nabla^k(-u\cdot\nabla u)\cdot\nabla^kudx+\int_{\mathbb{R}^n}\nabla^kv\cdot \nabla^kudx.
						\end{aligned}	
					\end{equation}
					By H$\ddot{\text{o}}$lder inequality, Sobolev inequality and Lemma \ref{lem1}, we can prove
					\begin{equation}\label{D10}
						\begin{aligned}
							&\Big|\int_{\mathbb{R}^n}\nabla^k(-u\cdot\nabla u)\cdot\nabla^kudx\Big|+\Big|\int_{\mathbb{R}^n}\nabla^kv\cdot \nabla^kudx\Big|\\
							&\leq C\|\nabla u\|_{L^\infty}\|\nabla^ku\|_{L^2}^2+\frac{1}{2}\|\nabla^ku\|_{L^2}^2+\frac{1}{2}\|\nabla^kv\|_{L^2}^2\\
							&\leq C\|\nabla^2 u\|_{H^{s-2}}\|\nabla^ku\|_{L^2}^2+\frac{1}{2}\|\nabla^ku\|_{L^2}^2+\frac{1}{2}\|\nabla^kv\|_{L^2}^2\\
							&\leq C\delta\|\nabla^ku\|_{L^2}^2+\frac{1}{2}\|\nabla^ku\|_{L^2}^2+\frac{1}{2}\|\nabla^kv\|_{L^2}^2\\
							&\leq \frac{3}{4}\|\nabla^ku\|_{L^2}^2+\frac{1}{2}\|\nabla^kv\|_{L^2}^2.
						\end{aligned}	
					\end{equation}
					The combination of \eqref{D9} and \eqref{D10} and summing $k$ from $1$ to $s+2$ yields
					\begin{equation}\label{D11}
						\frac{d}{dt}\|\nabla u\|_{H^{s+1}}^2+\frac{1}{2}\|\nabla u\|_{H^{s+1}}^2
						\leq\|\nabla v\|_{H^{s+1}}^2.		
					\end{equation}
					Integrating with time over $[0,t]$ yields 
					\begin{equation}\label{TT3}
						\|\nabla u(t)\|_{H^{s+1}}^2+\frac{1}{2}\int_0^t \|\nabla u(\tau)\|_{H^{s+1}}^2d\tau \leq \int_0^t\|\nabla v(\tau)\|_{H^{s+1}}^2d\tau+\delta_0^2.
					\end{equation}
					Taking $\eqref{TT2}+\frac{1}{2}\times\eqref{TT1}+\eqref{TT3}$ yields 
					\begin{equation}
						\begin{aligned}
							&\|v(t)\|_{H^{s+1}}^2+\frac{1}{2}\|(u-v)(t)\|_{L^2}^2+\int_0^t\|\nabla v(\tau)\|_{H^{s+1}}^2d\tau
							+\frac{1}{2}\int_0^t\|(u-v)(\tau)\|_{L^2}^2d\tau\\
							&+	\|\nabla u(t)\|_{H^{s+1}}^2+\frac{1}{2}\int_0^t \|\nabla u(\tau)\|_{H^{s+1}}^2d\tau\\
							&\leq C\delta \int_0^t\mathcal{D}(\tau)^2d\tau+C\delta_0^2.	
						\end{aligned}
					\end{equation}
					It follows from the definition of $\mathcal{D}(t)$ in \eqref{Dt} and the smallness of $\delta$ to prove 
					\begin{equation}\label{32701}
						\begin{aligned}
							\|u(t)\|_{H^{s+2}}^2+\|v(t)\|_{H^{s+1}}^2+\int_0^t(\| \nabla u(\tau)\|_{H^{s+1}}^2+\|\nabla v(\tau)\|_{H^{s+1}}^2+\|(u-v)(\tau)\|_{L^2}^2)d\tau
							\leq C\delta_0^2.
						\end{aligned}	
					\end{equation}
					The combination of 	\eqref{32701} and \eqref{32702} gives 
					$$
					\begin{aligned}
						\|\rho(t)\|_{H^s}^2+\|u(t)\|_{H^{s+2}}^2+\|v(t)\|_{H^{s+1}}^2+\int_0^t(\| \nabla u(\tau)\|_{H^{s+1}}^2+\|\nabla v(\tau)\|_{H^{s+1}}^2)d\tau
						\leq C\delta_0^2.
					\end{aligned}	
					$$
					In the next, to obtain the lower bound of $\rho(t,x)$, we define a backward characteristic $X(t,x)$ satisfies
					\begin{equation}
						\partial_\ell X(\ell,x)=u(\ell,X(\ell,x))\text{ with }X(t,x)=x.
					\end{equation}
					A straightforward calculation yields 
					\begin{equation}
						\frac{d}{ds}\rho(\ell,X(\ell,x))=-\rho(\ell,X(x,\ell)){\rm{div}}u(\ell,X(\ell,x)),
					\end{equation}
					which implies 
					\begin{equation}
						\rho(t,x)=\rho_0(X(0,x))\exp\Big\{-\int_0^t{\rm{div}}u(\tau,X(\tau,x))d\tau\Big\}.
					\end{equation}
					By Proposition \ref{p6}, we deduce that
					$$
					\begin{aligned}
						\Big|\int_0^t{\rm{div}}u(\tau,X(\tau,x))d\tau\Big|
						&\leq C\int_0^t\|\nabla u(\tau)\|_{L^\infty}d\tau\\
						&\leq C\int_0^t\|\nabla u(\tau)\|_{H^s}d\tau\\
						&\leq C(\int_0^t(1+\tau)^{-\frac{9}{8}}d\tau)^{\frac{1}{2}}(\int_0^t(1+\tau)^{\frac{9}{8}}\|\nabla u(\tau)\|_{H^s}^2d\tau)^{\frac{1}{2}}\\
						&\leq C(\mathcal{I}_0^2+\mathcal{I}_0^3)^{\frac{1}{2}}.
					\end{aligned}
					$$
					Since the initial density $\rho_0>0$ for $x\in \mathbb{R}^n$, it gives rise to
					\begin{equation}
						\rho(t,x)\geq \rho_0(X(0,x))e^{-C(\mathcal{I}_0^2+\mathcal{I}_0^3)^{\frac{1}{2}}}>0.
					\end{equation}
					This completes the proof of the proposition.
				\end{proof}
				
				\vskip 4mm
				\noindent{\it\textbf{Proof of Theorem \ref{Th1}.}\ }
				By Proposition \ref{p4}, we obtain 
				\begin{equation}
					\|\rho(t)\|_{H^s}^2+\|u(t)\|_{H^{s+2}}^2+\|v(t)\|_{H^{s+1}}^2+\int_0^t\left(\|\nabla u(\tau)\|_{H^{s+1}}^2+\|\nabla v(\tau)\|_{H^{s+1}}^2\right)d\tau
					\leq C\delta_0^2.
				\end{equation}
				Choosing the initial data $\delta_0$ sufficiently small such that $C\delta_0^2 \leq \frac{1}{4}\delta^2$, we obtain 
				\begin{equation}
					\sup_{t\geq 0}	\mathcal{Z}(t)=\sup_{t\geq 0 }\Big(\|\rho(t)\|_{H^s}^2+\|u(t)\|_{H^{s+2}}^2+\|v(t)\|_{H^{s+1}}^2\Big)^{\frac{1}{2}}\leq \frac{1}{2}\delta, 
				\end{equation}
				which, together with \eqref{rholower}, closes the a priori assumption $(\ref{a-priori est})$. Then, based on the continuous argument, the global existence of solution $(\rho,u,v)$ and the estimates \eqref{T1} are obtained.
				
				It remains to establish the large-time behavior of the solution $(u,v)$.
				By Proposition \ref{p2} and Proposition \ref{p3}, it is easy to verify for $s\geq 2[\frac{n}{2}]+1$ that 
				\begin{equation}
					\|(u,v)(t)\|_{L^2}\leq C\mathcal{I}_0(1+t)^{-\frac{n}{4}},\quad 
					\|\nabla(u,v)(t)\|_{H^s}\leq C\mathcal{I}_0(1+t)^{-\frac{n}{4}}.
				\end{equation}
				This completes the proof of Theorem \ref{Th1}.
				\endproof
				\section{Optimal decay rates of  the pressureless E-NS system}
				\label{S4}
				Indeed, we can prove the time decay rate of velocity $v$ in Proposition \ref{p2}  is optimal. This means that we shall establish the lower bound of the decay rates of $\|v\|_{L^2}$. 
				\begin{prop}\label{p5}
					Assume the same conditions in Theorem \ref{Th1} hold. There exists a sufficiently small constant $\delta_0>0$ such that if 
					\begin{equation} \label{041111}
						\|\rho_0\|_{H^s(\mathbb{R}^n)}+\|u_0\|_{H^{s+2}(\mathbb{R}^n)}+\|v_0\|_{H^{s+1}(\mathbb{R}^n)}+\|v_0\|_{L^1(\mathbb{R}^n)}\leq \delta_0,
					\end{equation}
					and the Fourier transform of the initial velocity $\hat{v}_0(\xi)$ satisfy 
					\begin{equation}
						\inf_{|\xi|\leq1}|\hat{v}_0(\xi)|\geq \delta_0^{\frac{3}{2}},
					\end{equation} 
					then it holds for large time that 
					\begin{equation}
						\|v(t)\|_{L^2} \geq \frac{1}{2}	\delta_0^{\frac{3}{2}}(1+t)^{-\frac{n}{4}}.
					\end{equation}
				\end{prop}
				\begin{proof}
					Since the assumption of  $\|v_0\|_{L^1}$ in \eqref{041111} is also sufficiently small, according to the result proved by Proposition 3.1 \cite{ChoiJung2021JMFM} , we have for any $\alpha \in (0,n/4)$, 
					\begin{equation}
						E(t)(1+t)^{2\alpha}+\int_0^t(1+\tau)^{2\alpha} D(\tau)d\tau\leq C(E(0)+\|v_0\|_{L^1}^2)\leq C\delta_0^2,
					\end{equation}
					which gives 
					\begin{equation}
						E(t)= \int_{\mathbb{R}^n} \rho|u|^2dx+\|v\|_{L^2}^2\leq C\delta_0^2(1+t)^{-2\alpha}, \quad \forall~ 0<\alpha<\frac{n}{4}.
					\end{equation}
					Choosing $\alpha=\frac{n}{4}-\frac{1}{16}$, we have 
					\begin{equation}\label{041102}
						E(t)\leq C\delta_0^2(1+t)^{-\frac{n}{2}+\frac{1}{8}}, \quad \|v(t)\|_{L^2}\leq C\delta_0 (1+t)^{-\frac{n}{4}+\frac{1}{16}}.
					\end{equation}
					Multiplying \eqref{newu} by $u$, integrating it in $\mathbb{R}^n$ and using the definition of $E(t)$ yield
					$$
					\begin{aligned}
						\frac{1}{2}\frac{d}{dt}\|u\|_{L^2}^2+\|u\|_{L^2}^2
						&=\int_{\mathbb{R}^n}(-u\cdot\nabla u+v)\cdot udx\\
						&\leq C\delta\|u\|_{L^2}^2+\frac{1}{4}\|u\|_{L^2}^2+\|v\|_{L^2}^2\\
						&\leq\frac{1}{2}\|u\|_{L^2}^2+C\delta_0^2(1+t)^{-\frac{n}{2}+\frac{1}{8}}.
					\end{aligned}
					$$
					Therefore, we have 
					$$
					\frac{1}{2}\frac{d}{dt}\|u\|_{L^2}^2+\frac{1}{2}\|u\|_{L^2}^2\leq C\delta_0^2(1+t)^{-\frac{n}{2}+\frac{1}{8}}.
					$$
					By Gr$\ddot{\text{o}}$nwall  inequality, it holds for $t\geq 0$ that 
					\begin{equation}\label{041112}
						\|u(t)\|_{L^2}\leq C\delta_0(1+t)^{-\frac{n}{4}+\frac{1}{16}}.
					\end{equation}
					We follow a similar method developed in Porposition \ref{p3} and use  \eqref{041102}, \eqref{041112} to show
					\begin{equation}\label{041201}
						\|\nabla(u,v)(t)\|_{H^s}\leq C\delta_0(1+t)^{-\frac{n}{4}+\frac{1}{16}}.
					\end{equation}
					By \eqref{31501} and \eqref{041102}, we have
					\begin{equation}\label{041103}
						\frac{d}{dt}\Big\{(1+t)^{\frac{9}{8}}E(t)\Big\}+(1+t)^{\frac{9}{8}}\|\nabla v\|_{L^2}^2\leq \frac{9}{8}(1+t)^{\frac{1}{8}}E(t)\leq C\delta_0^2(1+t)^{-\frac{n}{2}+\frac{1}{4}}.
					\end{equation}
					Noting that $n\geq 3$,  therefore integrating \eqref{041103} with time yields 
					\begin{equation}\label{041104}
						\int_0^t(1+\tau)^{\frac{9}{8}}\|\nabla v(\tau)\|_{L^2}^2d\tau\leq C\delta_0^2\int_0^t(1+\tau)^{-\frac{n}{2}+\frac{1}{4}}d\tau+E_0\leq C\delta_0^2.
					\end{equation}
					Taking $\int_{\mathbb{R}^n}\nabla \eqref{newu}\cdot\nabla udx$ and applying \eqref{041201} give rise to
					\begin{equation}
						\begin{aligned}
							\frac{1}{2}\frac{d}{dt}\|\nabla u\|_{L^2}^2+\|\nabla u\|_{L^2}^2
							&\leq \Big|\int_{\mathbb{R}^n}\nabla(-u\cdot\nabla u)\cdot\nabla udx\Big|+\Big|\int_{\mathbb{R}^n}\nabla u\cdot\nabla vdx\Big|\\
							&\leq C\|\nabla u\|_{L^\infty}\|\nabla u\|_{L^2}^2+C\|u\|_{L^\infty}\|\nabla^2u\|_{L^2}\|\nabla u\|_{L^2}
							+\|\nabla u\|_{L^2}\|\nabla v\|_{L^2}\\
							&\leq C\delta_0\|\nabla u\|_{L^2}^2+C\|u\|_{L^\infty}^2\|\nabla^2u\|_{L^2}^2+\frac{1}{4}\|\nabla u\|_{L^2}^2
							+ 2\|\nabla v\|_{L^2}^2\\
							&\leq  C\delta_0^4(1+t)^{-n+\frac{1}{4}}+\frac{1}{2}\|\nabla u\|_{L^2}^2+2\|\nabla v\|_{L^2}^2.	
						\end{aligned}	
					\end{equation}
					Thus, it is easy to show
					\begin{equation}\label{041105}
						\frac{d}{dt}\|\nabla u\|_{L^2}^2+\|\nabla u\|_{L^2}^2
						\leq  C\delta_0^4(1+t)^{-n+\frac{1}{4}}+4\|\nabla v\|_{L^2}^2.	
					\end{equation}
					Multiplying \eqref{041105} by $(1+t)^{9/8}$, intergrating it with time and using \eqref{041104}, one has   
					\begin{equation}\label{041106}
						\begin{aligned}
							\int_0^t(1+\tau)^{\frac{9}{8}}\|\nabla u(\tau)\|_{L^2}^2d\tau
							&\leq C\delta_0^4\int_0^t(1+\tau)^{-n+\frac{11}{8}}d\tau+4\int_0^t(1+\tau)^{\frac{9}{8}}\|\nabla v(\tau)\|_{L^2}^2d\tau+\|\nabla u_0\|_{L^2}^2\\
							&\leq C(\delta_0^2+\delta_0^4).
						\end{aligned}
					\end{equation}
					Multiply \eqref{u-v2} by $(u-v)$ and integrate over $\mathbb{R}^n$
					\begin{equation}
						\begin{aligned}
							\frac{1}{2}\frac{d}{dt}\|u-v\|_{L^2}^2+\|u-v\|_{L^2}^2=\int_{\mathbb{R}^n}
							(-u\cdot\nabla u+v\cdot\nabla v+\nabla P-\Delta v-\rho(u-v))\cdot (u-v)dx.		
						\end{aligned}
					\end{equation}
					It is easy to verify 
					$$
					\begin{aligned}
						&\Big| \int_{\mathbb{R}^n}
						(-u\cdot\nabla u+v\cdot\nabla v+\nabla P-\Delta v-\rho(u-v))\cdot (u-v)dx\Big|\\
						&\leq \Big| \int_{\mathbb{R}^n}
						(-u\cdot\nabla u+v\cdot\nabla v+\nabla P)\cdot (u-v)dx\Big|+
						\Big| \int_{\mathbb{R}^n}
						(-\Delta v-\rho(u-v))\cdot (u-v)dx\Big|\\
						&\leq (\|u\|_{L^\infty}\|\nabla u\|_{L^2}+\|v\|_{L^\infty}\|\nabla v\|_{L^2}+\|\nabla P\|_{L^2})\|u-v\|_{L^2}+\|\nabla v\|_{L^2}\|\nabla (u-v)\|_{L^2}+\|\rho\|_{L^\infty}\|u-v\|_{L^2}^2\\
						&\leq C(\|u\|_{L^\infty}\|\nabla u\|_{L^2}+\|v\|_{L^\infty}\|\nabla v\|_{L^2})\|u-v\|_{L^2}+\|\nabla v\|_{L^2}(\|\nabla u\|_{L^2}+\|\nabla v\|_{L^2})+C\|\rho\|_{L^\infty}\|u-v\|_{L^2}^2\\
						&\leq C\delta_0^4(1+t)^{-n+\frac{1}{4}}+\frac{3}{2}\|\nabla v\|_{L^2}^2+\frac{1}{2}\|\nabla u\|_{L^2}^2+\frac{1}{2}\|u-v\|_{L^2}^2.
					\end{aligned}
					$$
					Then, it holds 
					\begin{equation}\label{041107}
						\begin{aligned}
							\frac{d}{dt}\|u-v\|_{L^2}^2+\|u-v\|_{L^2}^2\leq 
							C\delta_0^4(1+t)^{-n+\frac{1}{4}}+3\|\nabla v\|_{L^2}^2+\|\nabla u\|_{L^2}^2.	
						\end{aligned}
					\end{equation}
					Multiplying \eqref{041107} by $(1+t)^{9/8}$, then integrating with time and using \eqref{041104}, \eqref{041106}, one has
					\begin{equation}\label{041108}
						\int_0^t	(1+\tau)^{\frac{9}{8}}\|(u-v)(\tau)\|_{L^2}^2d\tau\leq C(\delta_0^2+\delta_0^4)\leq C\delta_0^2 .
					\end{equation}   	
					Let $w(t,x)$ be the solution of the heat equation with the initial data $w(0,x)=v_0(x)$ satisfying
					\begin{equation}
						\left\{
						\begin{aligned}
							&	w_t-\Delta w=0,\\	
							&w(0,x)=v_0(x),\quad \text{div}v_0=0.
						\end{aligned}
						\right.
					\end{equation}
					Denote $q=v-w$. It is easy to prove
					$$
					\text{div}w=0,~~ \text{div}q=\text{div}v-\text{div}w=0.
					$$
					Then, the equation of $q$ is stated below
					\begin{equation}\label{O1}
						\left\{
						\begin{aligned}
							&q_t+v\cdot\nabla v+\nabla P=\Delta q+\rho(u-v),\\
							&\text{div}q=0,~~q(0,x)=0.
						\end{aligned}
						\right.
					\end{equation}
					Multiplying $\eqref{O1}_1$ by $q$ and integrating  over $\mathbb{R}^n$, we have
					\begin{equation}\label{O2}
						\begin{aligned}
							&\frac{1}{2}\frac{d}{dt} \|q\|_{L^2}^2+\|\nabla q\|_{L^2}^2 \\
							&=-\int_{\mathbb{R}^n}q\cdot(v\cdot\nabla v)dx-\int_{\mathbb{R}^n}q\cdot\nabla Pdx+\int_{\mathbb{R}^n}q\cdot \rho(u-v)dx\\
							&=\frac{1}{2}\int_{\mathbb{R}^n}\text{div}q\cdot|v|^2dx+\int_{\mathbb{R}^n}P\text{div}qdx+\int_{\mathbb{R}^n}(v-w)\cdot \rho(u-v)dx\\
							&=\int_{\mathbb{R}^n}v\cdot \rho(u-v)dx-\int_{\mathbb{R}^n}w\cdot \rho(u-v)dx.
						\end{aligned}
					\end{equation}
					Substituting \eqref{T21} into \eqref{O2} yields 
					\begin{equation}
						\begin{aligned}
							&\frac{1}{2}\frac{d}{dt}(\int_{\mathbb{R}^n}\rho|u|^2dx+\|q\|_{L^2}^2)+\int_{\mathbb{R}^n}\rho|u-v|^2dx+\|\nabla q\|_{L^2}^2\\
							&=-\int_{\mathbb{R}^n}w\cdot \rho(u-v)dx.
						\end{aligned}
					\end{equation}
					In terms of Cauchy inequality and Sobolev inequality, one has
					\begin{equation}
						\begin{aligned}
							&\frac{1}{2}\frac{d}{dt}\Big(\int_{\mathbb{R}^n}\rho|u|^2dx+\|q\|_{L^2}^2\Big)	+\int_{\mathbb{R}^n}\rho|u-v|^2dx+\|\nabla q\|_{L^2}^2\\
							&\leq \|w\|_{L^\infty}\int_{\mathbb{R}^n}\rho |u-v|dx\\
							&\leq \|w\|_{L^\infty}\|\rho\|_{L^1}^{\frac{1}{2}}
							\Big(\int_{\mathbb{R}^n}\rho|u-v|^2dx\Big)^{\frac{1}{2}}\\
							&\leq \|w\|_{L^\infty}\|\rho_0\|_{L^1}^{\frac{1}{2}}	\Big(\int_{\mathbb{R}^n}\rho|u-v|^2dx\Big)^{\frac{1}{2}}\\
							&\leq \frac{1}{2}\|w\|_{L^\infty}^2\|\rho_0\|_{L^1}+\frac{1}{2} \int_{\mathbb{R}^n}\rho|u-v|^2dx\\
							&\leq C\mathcal{I}_0^3(1+t)^{-n}+ \frac{1}{2}\int_{\mathbb{R}^n}\rho|u-v|^2dx,
						\end{aligned}
					\end{equation}
					where the last step have used the Lemma \ref{lem2} and the definition of $\mathcal{I}_0$ 
					\begin{equation}\label{winfty}
						\|w\|_{L^\infty}\leq C(1+t)^{-\frac{n}{2}}(\|v_0\|_{L^1}+\|\nabla v_0\|_{H^{s-2}})\leq C\mathcal{I}_0(1+t)^{-\frac{n}{2}}.
					\end{equation}
					It then concludes from the above inequalities to prove 
					\begin{equation}
						\frac{1}{2}\frac{d}{dt}\Big(\int_{\mathbb{R}^n}\rho|u|^2dx+\|q\|_{L^2}^2\Big)+\frac{1}{2}\int_{\mathbb{R}^n}\rho|u-v|^2dx+\|\nabla q\|_{L^2}^2	\leq C\mathcal{I}_0^3(1+t)^{-n}.
					\end{equation}
					With the help of the Plancherel Theorem, we verify 
					$$
					\begin{aligned}
						&\frac{1}{2}\frac{d}{dt}\Big(\int_{\mathbb{R}^n}\rho|u|^2dx+\|q\|_{L^2}^2\Big)
						+\int_{X_3(t)}|\xi|^2|\hat{q}(t,\xi)|^2d\xi+\int_{X_3^c(t)}|\xi|^2|\hat{q}(t,\xi)|^2d\xi\\
						&~~+\frac{1}{2}\int_{\mathbb{R}^n}\rho|u-v|^2dx\leq C\mathcal{I}_0^3(1+t)^{-n},
					\end{aligned}
					$$
					where $X_3(t)$ and $X_3^c(t)$  are defined by 
					\begin{equation}
						X_3(t)=\Big\{\xi: |\xi|\leq \Big(\frac{2n}{t+4n}\Big)^{\frac{1}{2}}\Big\},\quad 
						X_3^c(t)=\Big\{\xi: |\xi|>\Big(\frac{2n}{t+4n}\Big)^{\frac{1}{2}}\Big\}.
					\end{equation}
					As a result, we get 
					$$
					\begin{aligned}
						&\frac{1}{2}\frac{d}{dt}\Big(\int_{\mathbb{R}^n}\rho|u|^2dx+\|q\|_{L^2}^2\Big)
						+\frac{2n}{t+4n}\Big(\int_{\mathbb{R}^n}\rho|u-v|^2dx+\|q\|_{L^2}^2\Big)\\
						&\leq \frac{2n}{t+4n}\int_{X_3(t)}|\hat{q}(t,\xi)|^2d\xi
						+C\mathcal{I}_0^3(1+t)^{-n}.
					\end{aligned}
					$$
					It should be mentioned that 
					\begin{equation}\label{T31}
						\begin{aligned}
							\int_{\mathbb{R}^n}\rho|q-u|^2dx
							&\leq 2 \int_{\mathbb{R}^n}\rho|u-v|^2dx+2\int_{\mathbb{R}^n}\rho|w|^2dx\\
							&\leq  2 \int_{\mathbb{R}^n}\rho|u-v|^2dx+2\|\rho\|_{L^1}\|w\|_{L^\infty}^2\\
							&\leq  2 \int_{\mathbb{R}^n}\rho|u-v|^2dx+2\|\rho_0\|_{L^1}\|w\|_{L^\infty}^2.
						\end{aligned}
					\end{equation}
					Therefore, we obtain  
					\begin{equation}\label{O5}
						\int_{\mathbb{R}^n}\rho|u-v|^2dx+\|\rho_0\|_{L^1}\|w\|_{L^\infty}^2
						\geq \frac{1}{2}\int_{\mathbb{R}^n}\rho|q-u|^2dx.
					\end{equation}
					It then follows from above to prove 
					$$
					\begin{aligned}
						&\frac{1}{2}\frac{d}{dt}\Big(\int_{\mathbb{R}^n}\rho|u|^2dx+\|q\|_{L^2}^2\Big)
						+\frac{2n}{t+4n}\Big(\int_{\mathbb{R}^n}\rho|u-v|^2dx+\|q\|_{L^2}^2\Big)+\frac{2n}{t+4n}\|\rho_0\|_{L^1}\|w\|_{L^\infty}^2\\
						&\leq \frac{2n}{t+4n}\int_{X_3(t)}|\hat{q}(t,\xi)|^2d\xi+
						C\mathcal{I}_0^3(1+t)^{-n} +\frac{2n}{t+4n}\|\rho_0\|_{L^1}\|w\|_{L^\infty}^2 .
					\end{aligned}
					$$
					By \eqref{winfty}, we verify 
					\begin{equation}\label{O12}
						\begin{aligned}
							&\frac{1}{2}\frac{d}{dt}\Big(\int_{\mathbb{R}^n}\rho|u|^2dx+\|q\|_{L^2}^2\Big)
							+\frac{2n}{t+4n}\Big(\int_{\mathbb{R}^n}\rho|u-v|^2dx+\|\rho_0\|_{L^1}\|w\|_{L^\infty}^2+\|q\|_{L^2}^2\Big)\\
							&\leq \frac{2n}{t+4n}\int_{X_3(t)}|\hat{q}(t,\xi)|^2d\xi+	C\mathcal{I}_0^3(1+t)^{-n}.
						\end{aligned}
					\end{equation}
					After a direct calculation, one has
					\begin{equation}
						\begin{aligned}
							\int_{\mathbb{R}^n}\rho|u|^2dx +\|q\|_{L^2}^2
							&\leq  2\int_{\mathbb{R}^n}\rho|q-u|^2dx+2\int_{\mathbb{R}^n}\rho|q|^2dx+\|q\|_{L^2}^2\\
							&\leq  2\int_{\mathbb{R}^n}\rho|q-u|^2dx+2(\|\rho\|_{L^\infty}+\frac{1}{2})\|q\|_{L^2}^2\\
							&\leq 2(\int_{\mathbb{R}^n}\rho|q-u|^2dx+\|q\|_{L^2}^2).
						\end{aligned}
					\end{equation}
					Thus we obtain
					\begin{equation}\label{O6}
						\Big(\int_{\mathbb{R}^n}\rho|q-u|^2dx+\|q\|_{L^2}^2\Big)\geq \frac{1}{2}\Big(\int_{\mathbb{R}^n}\rho|u|^2dx+\|q\|_{L^2}^2\Big).
					\end{equation}
					The combination of \eqref{O5} and \eqref{O6} yields 
					$$
					\Big(\int_{\mathbb{R}^n}\rho|u-v|^2dx+\|\rho_0\|_{L^1}\|w\|_{L^\infty}^2+\|q\|_{L^2}^2\Big)\geq\frac{1}{4}
					\Big(\int_{\mathbb{R}^n}\rho|u|^2dx+\|q\|_{L^2}^2\Big).
					$$
					Therefore, \eqref{O12} can be expressed below
					\begin{equation}
						\begin{aligned}
							&\frac{1}{2}\frac{d}{dt}\Big(\int_{\mathbb{R}^n}\rho|u|^2dx+\|q\|_{L^2}^2\Big)+\frac{n}{2(t+4n)}\Big(\int_{\mathbb{R}^n}\rho|u|^2dx+\|q\|_{L^2}^2\Big)\\
							&\leq \frac{2n}{t+4n}\int_{X_3(t)}|\hat{q}(t,\xi)|^2d\xi+C\mathcal{I}_0^3(1+t)^{-n}.
						\end{aligned}
					\end{equation}
					For simplicity, we also define the energy $H(t)$ as follows
					$$
					H(t)=\int_{\mathbb{R}^n}\rho|u|^2dx+\|q\|_{L^2}^2,~~H(0)=\int_{\mathbb{R}^n}\rho_0|u_0|^2dx.
					$$
					It is easy to verify that 
					\begin{equation}\label{O3}
						\frac{d}{dt} H(t)+\frac{n}{t+4n}H(t)
						\leq C(1+t)^{-1}\int_{X_3(t)}|\hat{q}(t,\xi)|^2d\xi+C\mathcal{I}_0^3(1+t)^{-n}.
					\end{equation}
					The next goal is to calculate $\int_{X_3(t)}|\hat{q}(t,\xi)|^2d\xi$. 
					$$
					q_t+v\cdot\nabla v+\nabla P=\Delta q+\rho(u-v).
					$$
					According to the definition of \eqref{T7}, the above equation can be stated below 
					\begin{equation}
						\left\{
						\begin{aligned}
							&q_t-\Delta q=F,\\
							&q(0,x)=0.
						\end{aligned}
						\right.
					\end{equation}
					Taking Fourier transform and using Duhamel's principle, we have
					$$
					\hat{q}(t,\xi)=\int_0^te^{-|\xi|^2(t-\tau)}\hat{F}(\tau,\xi)d\tau.
					$$
					In terms of \eqref{L2} and Proposition \ref{p2}, we can prove 
					$$
					|\hat{F}(\tau,\xi)|\leq |\xi|\|v\|_{L^2}^2+\|\rho(u-v)\|_{L^1}\leq C|\xi|\mathcal{I}_0^2(1+\tau)^{-\frac{n}{2}}+\|\rho(u-v)\|_{L^1}. 
					$$
					It follows from H$\ddot{\text{o}}$lder inequality and \eqref{041108} to prove
					$$
					\begin{aligned}
						\int_0^t\|\rho(u-v)(\tau)\|_{L^1}d\tau
						&\leq	\int_0^t \|\rho(\tau)\|_{L^2}\|(u-v)(\tau)\|_{L^2}d\tau\\
						&\leq C\delta_0\int_0^t\|(u-v)(\tau)\|_{L^2}d\tau\\
						&\leq C\delta_0(\int_0^t(1+\tau)^{-\frac{9}{8}}d\tau)^{\frac{1}{2}}
						(\int_0^t(1+\tau)^{\frac{9}{8}}\|(u-v)(\tau)\|_{L^2}^2d\tau)^{\frac{1}{2}}\\
						&\leq C\delta_0^2.
					\end{aligned}
					$$
					Hence, we get 
					\begin{equation}
						\begin{aligned}
							|\hat{q}(t,\xi)|
							&\leq C|\xi|\mathcal{I}_0^2\int_0^t e^{-|\xi|^2(t-\tau)}(1+\tau)^{-\frac{n}{2}}d\tau+\int_0^t\|\rho(u-v)(\tau)\|_{L^1}d\tau\\
							&\leq C|\xi|\mathcal{I}_0^2e^{-\frac{1}{2}|\xi|^2t}+C|\xi|^{-1}\mathcal{I}_0^2(1+t)^{-\frac{n}{2}}+C\delta_0^2.
						\end{aligned}
					\end{equation}
					After that, it holds
					$$
					\begin{aligned}
						\int_{X_3(t)} |\hat{q}(t,\xi)|^2d\xi\leq C\mathcal{I}_0^4(1+t)^{-\frac{n+2}{2}}+C\delta_0^4(1+t)^{-\frac{n}{2}}.
					\end{aligned}
					$$
					In terms of \eqref{O3}, it holds 
					\begin{equation}\label{O4}
						\begin{aligned}
							&\frac{d}{dt}H(t)+\frac{n}{t+4n}H(t)\\
							&\leq C\mathcal{I}_0^4(1+t)^{-\frac{n+4}{2}}+
							C\delta_0^4(1+t)^{-\frac{n+2}{2}}+C\mathcal{I}_0^3(1+t)^{-n}.
						\end{aligned}
					\end{equation}
					Then, multiplying \eqref{O4} by $(t+4n)^n$ and integrating the resulting equation with time gives  
					$$
					\begin{aligned}
						H(t)&\leq C\mathcal{I}_0^4(1+t)^{-\frac{n+2}{2}}+H(0)(1+t)^{-n}\\
						&+C\delta_0^4(1+t)^{-\frac{n}{2}}+C\mathcal{I}_0^3(1+t)^{-n+1}.
					\end{aligned}
					$$
					Hence, we obtain the following decay rate for large time that 
					\begin{equation}\label{OT1}
						\begin{aligned}
							\|q(t)\|_{L^2}&\leq C\mathcal{I}_0^2(1+t)^{-\frac{n+2}{4}}+C\delta_0(1+t)^{-\frac{n}{2}}+
							C\delta_0^2(1+t)^{-\frac{n}{4}}\\
							&\quad+C\mathcal{I}_0^{\frac{3}{2}}(1+t)^{-\frac{n-1}{2}}\\
							&\leq C\delta_0^2(1+t)^{-\frac{n}{4}}.
						\end{aligned}
					\end{equation}
					In terms of Lemma \ref{lem3}, \eqref{OT1} and the smallness of $\delta_0$, we can prove that for large time that
					\begin{equation}\label{smallt0}
						\begin{aligned}
							\|v\|_{L^2}&\geq \|w\|_{L^2}-\|q\|_{L^2}\\
							&\geq \delta_0^{\frac{3}{2}} (1+t)^{-\frac{n}{4}}-C\delta_0^2(1+t)^{-\frac{n}{4}}\\
							&\geq  \frac{1}{2}\delta_0^{\frac{3}{2}} (1+t)^{-\frac{n}{4}}.
						\end{aligned}
					\end{equation}
					This completes the proof of this proposition.
				\end{proof}
				\noindent{\it\textbf{Proof of Theorem \ref{Th2}.}\ }
				By Proposition \ref{p2}, we have the upper bound time decay rate of the velocity, 
				\begin{equation}
					\|v(t)\|_{L^2}\leq C\mathcal{I}_0(1+t)^{-\frac{n}{4}}.
				\end{equation}
				According to Proposition \ref{p5}, we obtain the lower bound decay rate of velcity  for large time
				\begin{equation}
					\|v(t)\|_{L^2} \geq \frac{1}{2}	\delta_0^{\frac{3}{2}}(1+t)^{-\frac{n}{4}}.
				\end{equation}
				Thus, we combine them together to prove 
				\begin{equation}
					\begin{aligned}
						\frac{1}{2}		\delta_0^{\frac{3}{2}}(1+t)^{-\frac{n}{4}}\leq \|v(t)\|_{L^2} \leq C\mathcal{I}_0(1+t)^{-\frac{n}{4}}.
					\end{aligned}
				\end{equation}
				This completes the proof of Theorem \ref{Th2}.
				\endproof
				
				{\bf Acknowledgments.}  The work of the first author was partially supported by National Key R\&D Programof China 2021YFA1000800, and National Natural Sciences Foundation of China (NSFC) 11688101. The work of the third author was supported by  National Natural Science Foundation of China (NSFC)
				12001033.

			\end{document}